\documentclass[11pt,reqno]{amsart}
\usepackage[T1]{fontenc}
\usepackage[utf8]{inputenc}
\usepackage{lmodern}
\usepackage{amsmath,amssymb,amsthm,mathtools}
\usepackage[a4paper,margin=1.12in]{geometry}
\usepackage{comment}

\usepackage{tikz}
\usepackage[dvipsnames]{xcolor}
\newlength{\PlotWidth}
\usetikzlibrary{arrows.meta}
\usepackage{microtype}
\usepackage{aliascnt}
\usepackage[hidelinks]{hyperref}
\usepackage[nameinlink,noabbrev]{cleveref}
\hypersetup{
  pdftitle={Temporal central limit theorem for mixed-radix odometer discrepancy},
  pdfauthor={Maria Avdeeva and Francesco Cellarosi}
}

\theoremstyle{plain}
\newtheorem{theorem}{Theorem}[section]
\newaliascnt{proposition}{theorem}
\newtheorem{proposition}[proposition]{Proposition}
\aliascntresetthe{proposition}
\newaliascnt{lemma}{theorem}
\newtheorem{lemma}[lemma]{Lemma}
\aliascntresetthe{lemma}
\newaliascnt{corollary}{theorem}
\newtheorem{corollary}[corollary]{Corollary}
\aliascntresetthe{corollary}
\theoremstyle{definition}
\newaliascnt{definition}{theorem}

\aliascntresetthe{definition}
\newaliascnt{example}{theorem}
\newtheorem{example}[example]{Example}
\aliascntresetthe{example}
\theoremstyle{remark}
\newaliascnt{remark}{theorem}
\newtheorem{remark}[remark]{Remark}
\aliascntresetthe{remark}

\crefname{theorem}{theorem}{theorems}
\Crefname{theorem}{Theorem}{Theorems}
\crefname{proposition}{proposition}{propositions}
\Crefname{proposition}{Proposition}{Propositions}
\crefname{lemma}{lemma}{lemmas}
\Crefname{lemma}{Lemma}{Lemmas}
\crefname{corollary}{corollary}{corollaries}
\Crefname{corollary}{Corollary}{Corollaries}
\crefname{remark}{remark}{remarks}
\Crefname{remark}{Remark}{Remarks}
\crefname{example}{example}{examples}
\Crefname{example}{Example}{Examples}

\newcommand{\R}{\mathbb R}
\newcommand{\E}{\mathbb E}
\newcommand{\Pp}{\mathbb P}
\newcommand{\Var}{\operatorname{Var}}
\newcommand{\Cov}{\operatorname{Cov}}
\newcommand{\1}{\mathbf 1}
\newcommand{\cD}{\mathcal D}
\newcommand{\cF}{\mathcal F}
\newcommand{\cG}{\mathcal G}
\newcommand{\Normal}{\mathcal N}
\newcommand{\dK}{d_{\mathrm{KS}}}
\newcommand{\TV}{d_{\mathrm{TV}}}

\newcommand{\eps}{\varepsilon}

\title{Temporal central limit theorem for mixed-radix odometer discrepancy}
\author{Maria Avdeeva}
\address[M. Avdeeva]{Center for Computational Biology, Flatiron Institute, Simons Foundation,
New York, NY 10010, USA.}
\email{mavdeeva@flatironinstitute.org}
\author{Francesco Cellarosi}
\address[F. Cellarosi]{Department of Mathematics and Statistics, Queen's University. 48 University Ave.
Jeffery Hall
Kingston, ON K7L 3N6, Canada.}
\email{francesco.cellarosi@queensu.ca}
\date{}

\keywords{Temporal central limit theorem, mixed-radix odometers,
local discrepancy, van der Corput sequences, quantitative normal
approximation, reverse martingales}

\subjclass[2020]{11K38, 60F05, 37A50, 37A44.}

\begin{document}
\begin{abstract}
    We consider  
    the local discrepancy for the number of visits of the orbit of $0$ under a  mixed-radix odometer to intervals near the origin. Both the length of the orbit and the endpoint of the visited interval are chosen independently and uniformly at random. 
    Under a Feller-type condition 
    and an explicit condition ensuring that the finite random-time clock can be approximated by a Haar clock on the fluctuation scale, we obtain a temporal central limit theorem for the discrepancy, with an explicit Kolmogorov–Smirnov rate of convergence. This extends and strengthens the fixed-base temporal central limit theorem of Levin and Merzbach to the mixed-radix setting.
\end{abstract}

\maketitle

\section{Introduction and main results}\label{sec:intro}
Fix an integer sequence $\mathbf{b}=(b_1, b_2, b_3,\ldots)$ with $b_j\geq2$.
Consider the compact abelian group 
$G= \prod_{j\geq1}\mathbb Z/b_j\mathbb Z$, whose elements are of the form $g=(g_1,g_2,g_3,\ldots)$ with $g_j\in\{0,1,\ldots,b_j-1\}$.
 Let $T$ be the \emph{generalized von Neumann - Kakutani adding machine} (or simply \emph{odometer}) given by the addition by $(1,0,0,\ldots)$ with carry-over on $G$, which preserves the Haar measure.
  
 The map $T$ can be realized up to measurable isomorphism as a cutting-and-stacking transformation on the interval $[0,1)$, preserving the Lebesgue measure, as follows. The interval $[0,1)$ is cut into $b_1$ equal intervals $\left[\frac{g_1}{b_1},\frac{g_1+1}{b_1}\right)$, where $0\leq g_1\leq b_1-1$, which are labelled by $g_1$.  The $g_1$-th interval is then stacked on top of the $(g_1-1)$-st interval, forming a tower consisting of $b_1$ floors. The tower is then cut into $b_2$ equal pieces which are labelled by $g_2\in\{0,\ldots,b_2-1\}$. The new labels are appended to the previous labels. Each piece of the tower is then stacked onto the one on its left, thus forming a new tower of $b_1b_2$ floors. The procedure of cutting, labelling, and stacking continues indefinitely, by cutting the tower built at the $j$-th step (and consisting of $b_1b_2\cdots b_{j}$ floors) into $b_{j+1}$ pieces, and so on. At any finite stage of the cutting-and-stacking construction, the map $T$ is defined on the tower except at the last floor: one  iteration of the map $T$ consists of moving upward by one floor. The map $T$ can be written as a piecewise isometry of the interval, with countably many branches, where the jumps occur at the points $1-\frac{1}{b_1\cdots b_j}$, $j\geq1$. The action of $T$ on the labels of each sub-interval is precisely the addition of $(1,0,0,\ldots)$ with carry-over  on $G$. See Figure \ref{fig:cutting-stacking} for an example with $\mathbf{b}=(2,3,4,\ldots)$. Every real number $x\in[0,1)$ corresponds to a unique infinite sequence of labels $(g_1,g_2,g_3,\ldots)$, coming from the sequence of nested intervals from the above construction, shrinking down toward $\{x\}$. If we define \begin{align}
B_0=1,\hspace{.3cm} B_j:=\prod_{\ell=1}^j b_\ell\quad(j\geq1),\label{eq:B-def}
\end{align}
then we write $x\in[0,1)$ as
\begin{align}
x=\sum_{j=1}^\infty \frac{g_j}{B_j},\hspace{.3cm}g_j\in\{0,1,\ldots, b_j-1\}.\label{B-expansion-real-x}
\end{align}

At points having two such expansions, we use the one selected by the left-closed, right-open intervals in the cutting-and-stacking construction. Namely, 
we exclude expansions whose digits are eventually maximal (i.e. for which $g_j=b_j-1$ for every sufficiently large $j$). We  refer to \eqref{B-expansion-real-x} as the \emph{mixed-radix} (or \emph{Cantor}) \emph{expansion of $x\in[0,1)$}. Thus, the interval map $T$ acts on points \eqref{B-expansion-real-x} by addition with carry-over of $(1,0,0,\ldots)$ to their mixed-radix digits. The the odometer is minimal and uniquely ergodic, with Haar measure as its unique invariant probability measure; see \cite[Sec.~1]{Downarowicz}. 
 The case $\mathbf{b}=(q, q, q, \ldots)$ corresponds to the usual base-$q$ expansion of $x$ (with the same convention of excluding expansions that are eventually equal to $q-1$) and the transformation $T$ is the classical von Neumann--Kakutani $q$-adic adding machine.

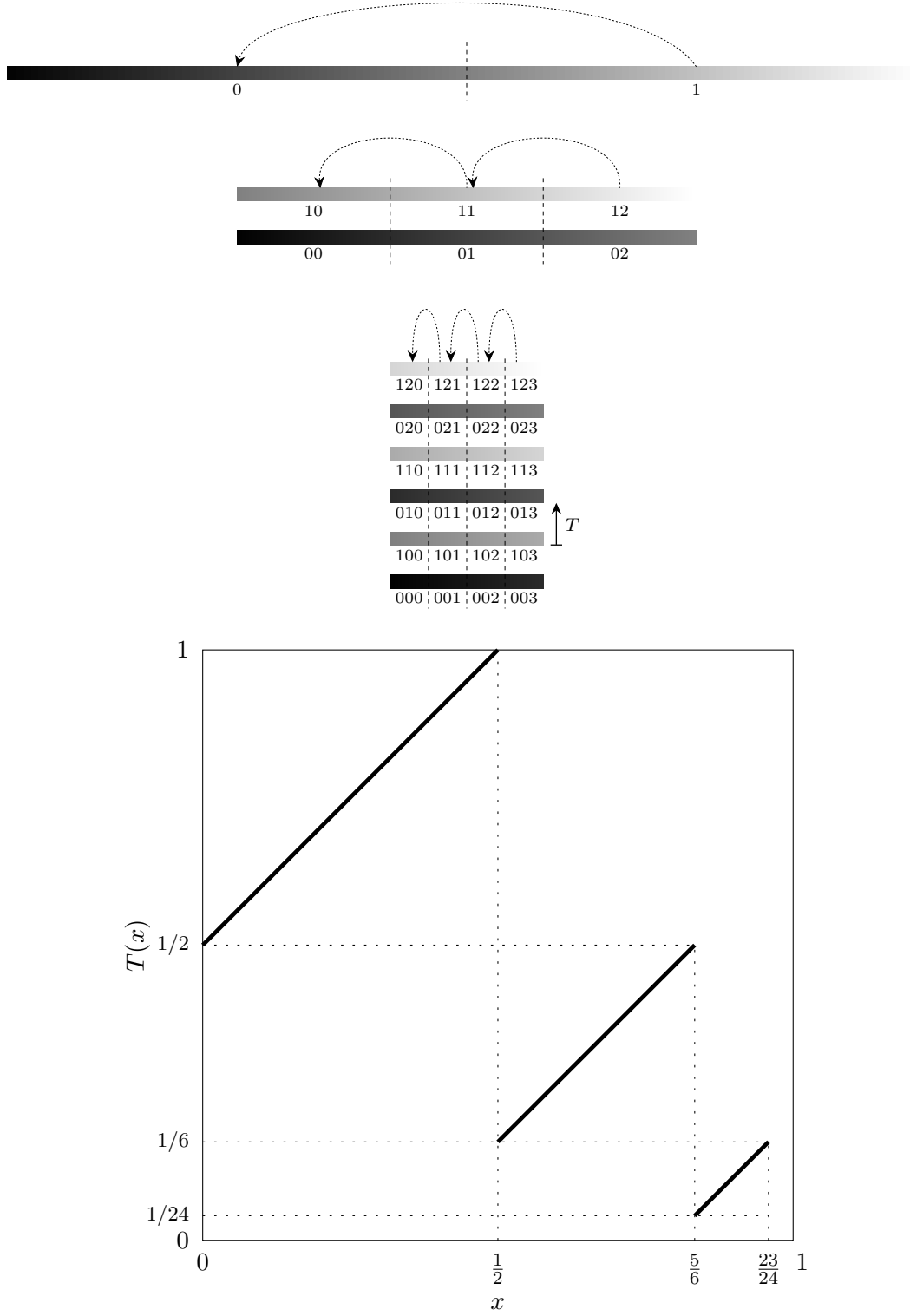
\begin{figure}[htbp]
  \centering
 \begin{tikzpicture}[
    >={Stealth[length=1.7mm,width=1.5mm]},
    font=\tiny,
    cut/.style={dash pattern=on 1.8pt off 1.8pt, line width=0.35pt},
    map/.style={->, dash pattern=on .8pt off 0.8pt, line width=0.3pt}
  ]
 
  \pgfmathsetmacro{\W}{14}       
  \pgfmathsetmacro{\h}{0.20}    
  \pgfmathsetmacro{\laby}{-0.15}
  \pgfmathsetmacro{\pB}{0.65}   
  \pgfmathsetmacro{\pC}{0.65}   
 
 
  \begin{scope}
    \shade[left color=black, right color=white] (0,0) rectangle (\W,\h);
    \draw[cut] (0.5*\W, \h+0.38) -- (0.5*\W, \laby-0.16);
    \node at (0.25*\W, \laby) {0};
    \node at (0.75*\W, \laby) {1};
    \draw[map] (0.75*\W, \h)
      .. controls +(125:1.7) and +(75:1.2) .. (0.25*\W, \h);
  \end{scope}
 
  \begin{scope}[shift={(\W/4,-2.5)}]
    \pgfmathsetmacro{\w}{\W/2}   
    \pgfmathsetmacro{\c}{\w/3}   
    \foreach \lvl/\gl/\gr in {0/100/50, 1/50/0} {
      \shade[left color=black!\gl, right color=black!\gr]
        (0, \lvl*\pB) rectangle (\w, \lvl*\pB+\h);
      \foreach \j in {0,1,2}
        \node at ({(\j+0.5)*\c}, {\lvl*\pB+\laby}) {\lvl\j};
    }
    \foreach \j in {1,2}
      \draw[cut] (\j*\c, \pB+\h+0.15) -- (\j*\c, \laby-0.16);
    \foreach \j in {1,2}
      \draw[map] ({(\j+0.5)*\c}, \pB+\h)
        .. controls +(85:1.0) and +(95:1.0) .. ({(\j-0.5)*\c+.1}, \pB+\h);
  \end{scope}
 
  \begin{scope}[shift={(\W/6+\W/4,-7.75)}]
    \pgfmathsetmacro{\w}{\W/6}   
    \pgfmathsetmacro{\c}{\w/4}   
    \foreach \lvl/\pref/\s in {0/00/0, 1/10/3, 2/01/1, 3/11/4, 4/02/2, 5/12/5} {
      \pgfmathsetmacro{\gl}{100*(1-\s/6)}
      \pgfmathsetmacro{\gr}{100*(1-(\s+1)/6)}
      \shade[left color=black!\gl, right color=black!\gr]
        (0, \lvl*\pC) rectangle (\w, \lvl*\pC+\h);
      \foreach \j in {0,...,3}
        \node at ({(\j+0.5)*\c}, {\lvl*\pC+\laby}) {\pref\j};
    }
    \foreach \j in {1,2,3}
      \draw[cut] (\j*\c, 5*\pC+\h+0.05) -- (\j*\c, \laby-0.16);
    \foreach \j in {1,2,3}
      \draw[map] ({(\j+0.3)*\c}, 5*\pC+\h)
        .. controls +(90:1.05) and +(90:1.05) .. ({(\j-0.5)*\c+.05}, 5*\pC+\h);
    \draw[|->, line width=0.5pt] (\w+0.2, 0.65) -- ++(0,0.65)
      node[right, pos=0.5, font=\scriptsize] {$T$};
  \end{scope}
\end{tikzpicture}
\par\vspace{.3cm}
\begin{tikzpicture}[
  x=\PlotWidth,
  y=\PlotWidth,
  font=\sffamily,
  guide/.style={
    thin,
    dash pattern=on 1pt off 4pt
  },
  branch/.style={
    black,
    line width=1.8pt,
    line cap=butt
  }
]

  \draw[guide] (0,{1/2}) -- ({5/6},{1/2});
  \draw[guide] (0,{1/6}) -- ({23/24},{1/6});
  \draw[guide] (0,{1/24}) -- ({23/24},{1/24});

  \draw[guide] ({1/2},0) -- ({1/2},1);
  \draw[guide] ({5/6},0) -- ({5/6},{1/2});
  \draw[guide] ({23/24},0) -- ({23/24},{1/6});

  \draw[branch] (0,{1/2}) -- ({1/2},1);
  \draw[branch] ({1/2},{1/6}) -- ({5/6},{1/2});
  \draw[branch] ({5/6},{1/24}) -- ({23/24},{1/6});

  \draw[thin] (0,0) rectangle (1,1);

  \node[below=2pt] at (0,0) {$0$};
  \node[below=2pt] at ({1/2},0) {$\tfrac{1}{2}$};
  \node[below=2pt] at ({5/6},0) {$\tfrac{5}{6}$};
  \node[below=2pt] at ({23/24},0) {$\tfrac{23}{24}$};
  \node[below=2pt,xshift=4pt] at (1,0) {$1$};

  \node[left=2pt] at (0,0) {$0$};
  \node[left=2pt] at (0,{1/24}) {{\footnotesize $1/24$}};
  \node[left=2pt] at (0,{1/6}) {{\footnotesize$1/6$}};
  \node[left=2pt] at (0,{1/2}) {{\footnotesize$1/2$}};
  \node[left=2pt] at (0,1) {$1$};

  \node[below=16pt] at ({1/2},-.02) {$x$};
  \node[rotate=90,above=19pt] at (0,{1/2}) {$T(x)$};
\end{tikzpicture}
 \caption{Top: Illustration of the first steps of the cutting-and-stacking
  procedure for $b_1=2$, $b_2=3$, and $b_3=4$. These first three steps
  define the map $T$ everywhere except on the interval
  $[\tfrac{23}{24},1)$ (which becomes the top floor when the four
sub-towers are stacked), where its action is determined by the
  subsequent cutting and stacking corresponding to $b_4,b_5,\ldots$.
  The label of each subinterval gives the initial digits of the
  mixed-radix expansion \eqref{B-expansion-real-x} shared by all points in that subinterval.
  Bottom: the (partial) graph of $T$ corresponding to the above cutting and stacking. }
  \label{fig:cutting-stacking}
\end{figure}

Analogously to \eqref{B-expansion-real-x}, every nonnegative integer $n$ has a unique mixed-radix expansion
\begin{equation}\label{digits_of_integers}
 n=\sum_{i\geq0}c_i(n)B_i,
 \qquad  c_i(n)\in\{0,1,\ldots, b_{i+1}-1\}.
\end{equation}
The orbit of the origin in the interval model of $T$ is then given explicitly in terms of \eqref{digits_of_integers} by
\begin{equation}\label{eq:vdc}
 x_n:=T^n0=\sum_{i\geq0}\frac{c_i(n)}{B_{i+1}}.
\end{equation}
The rational sequence \eqref{eq:vdc} generalizes the classical van der
Corput sequence to mixed-radix expansions. For instance, when $\mathbf{b}=(2,3,4,5,\ldots)$, we have $(x_n)_{n\ge0}=(0,\frac{1}{2},\frac{1}{6},\frac{2}{3},\frac{1}{3},\frac{5}{6},\frac{1}{24},\frac{13}{24},\frac{5}{24},$ $\frac{17}{24},\frac{3}{8},\frac{7}{8},\frac{1}{12},\frac{7}{12},\frac{1}{4},\frac{3}{4},\frac{5}{12},\frac{11}{12},\frac{1}{8},\frac{5}{8},\frac{7}{24},\ldots)$.
Our goal is to study the statistics of visits of the sequence $(x_n)_{n\geq0}$ to an interval of the form $[0,y)$. To this end, for $N\geq0$ and $y\in[0,1)$, define the local discrepancy
\begin{equation}\label{eq:D-N}
 D_N(y):=\sum_{n=0}^{N-1}\bigl(\1_{[0,y)}(x_n)-y\bigr),
 \qquad D_0(y)=0.
\end{equation}
Let $Y,Z$ be independent uniform random variables on $[0,1)$. For an integer $H\geq1$, called the \emph{(deterministic) time horizon}, define the random variables
\begin{equation}\label{eq:random-statistic}
 N_H:=\lfloor HZ\rfloor,\qquad \cD_H:=D_{N_H}(Y).
\end{equation}
Thus $N_H$ is a uniformly sampled time in $\{0,\ldots,H-1\}$ and is independent of $Y$.  We subsequently refer to $N_H$ simply as the \emph{random time}.  We call $c(N_H):=(c_i(N_H))_{i\geq0}$ its \emph{clock string}.
Our main results concern the distribution of $\cD_H$ as $H\to\infty$.
The mixed-radix expansions \eqref{B-expansion-real-x} and \eqref{digits_of_integers} of $Y$ and $N_H$, respectively, will be a crucial tool in our analysis. 
We stress that in the local discrepancy $\cD_H$ the random time $N_H$ and the endpoint $Y$ of the interval $[0,Y)$ are jointly randomised.
For a real random variable $X$, we write its Kolmogorov-Smirnov distance from a standard normal $\Normal(0,1)$ as
\begin{equation}\label{eq:v71-auto-0001}
 \dK(X,\Normal(0,1))
 :=\sup_{t\in\R}|\Pp(X\leq t)-\Phi(t)|
\end{equation}
where $\Phi(t)=\frac{1}{\sqrt{2\pi}}\int_{-\infty}^t e^{-u^2/2}\,\mathrm{d} u$. Since the law of $\Normal(0,1)$ has no atoms, convergence to $\Normal(0,1)$ in the metric $\dK$ is equivalent to convergence in law.

In order to state our main theorem, we need to introduce several deterministic sequences. For $m\geq1$ we set
\begin{equation}\label{eq:SQDelta}
 S_m:=\sum_{j=1}^m b_j^2,
 \qquad Q_m:=\sum_{j=1}^m b_j^4,
 \qquad \Delta_m:=\frac{Q_m}{S_m^2}.
\end{equation}
Moreover, let
\begin{align}
\mu(b):=\frac{b^2-1}{12b},\qquad \tau^2(b):=\frac{3b^4+10b^2-13}{720b^2},\label{eq:mu_and_tau}
\end{align}
and, for $m\geq1$,
\begin{align}
 M_m:=\sum_{j=1}^m\mu(b_j),\qquad T_m:=\sum_{j=1}^m\tau^2(b_j),\qquad V_m:=T_m+\frac{1-B_m^{-1}}{12}.\label{eq:M_and_T_and_V}
\end{align}

Also, for an integer $r\geq1$ we consider the remainder $s$ of the division of $H$ by $B_r$, i.e. we  write $H= k B_r+s$ with $0\leq s<B_r$, and use it to define\footnote{The definition \eqref{eq:epsilon-time} is motivated probabilistically via couplings in \Cref{rem:maximal-coupling}.}
\begin{equation}\label{eq:epsilon-time}
 \eps_{H,r}:=\frac{s(B_r-s)}{H B_r}.
\end{equation} Then, for $p\geq1$ and integers $1\le r\le m$, we use \eqref{eq:epsilon-time} to set
\begin{equation}\label{eq:rho-p}
 \rho_{H,r,m}^{(p)}:=
 \frac{1+\sum_{j=r}^m b_j+
       \bigl(\sum_{j=1}^m b_j\bigr)\eps_{H,r}^{1/p}}
      {\sqrt{S_m}}.
\end{equation}

For $H\geq2$, let $m(H):=\min\{r\geq1:H\leq B_r\}$
so that 
\begin{equation}\label{H-sandwiched}
B_{m(H)-1}<H\leq B_{m(H)}.
\end{equation}
Note that $H\to\infty$ if and only if $m(H)\to\infty$. 
Our main result is the following temporal central limit theorem with explicit rates.
\begin{theorem}[Quantitative temporal CLT]\label{thm:main}
There is an absolute constant $C$ such that, for every $H\geq2$ and every
$1\leq r\leq m:=m(H)$, we have
\begin{align}
 |\E[\cD_H]-M_m|
 &\leq C\sqrt{S_m}\,\rho_{H,r,m}^{(2)},\label{eq:mean-bound}\\
 \bigl|\sqrt{\Var(\cD_H)}-\sqrt{V_m}\bigr|
 &\leq C\sqrt{S_m}\,\rho_{H,r,m}^{(2)},\label{eq:sd-bound}\\
 \dK\left(\frac{\cD_H-M_m}{\sqrt{V_m}}\:,\:\Normal(0,1)\right)
 &\leq C\left(\Delta_m^{1/5}
       +\bigl(\rho_{H,r,m}^{(2)}\bigr)^{2/3}\right).
 \label{eq:main-rate}
\end{align}
Consequently, along any sequence $H\to\infty$ for which
\begin{equation}\label{eq:main-conditions}
 \Delta_{m(H)}\to0
 \quad\text{and}\quad
 \inf_{1\leq r\leq m(H)}\rho_{H,r,m(H)}^{(2)}\to0,
\end{equation}
we have, with $m=m(H)$,
\begin{align}
 \E[\cD_H]&=M_m+o(\sqrt{S_m}),\label{eq:main-mean}\\
 \Var(\cD_H)&=V_m+o(S_m),\label{eq:main-var}\\
 \frac{\cD_H-M_m}{\sqrt{V_m}}
 &\Longrightarrow\Normal(0,1).\label{eq:main-explicit}
\end{align}
\end{theorem}
The next theorem also provides convergence of moments.
\begin{theorem}[Convergence of moments]\label{thm:moments}
Assume that \eqref{eq:main-conditions} holds along a sequence $H\to\infty$.  Fix $p>0$ and suppose there exist
$q>\max\{p,2\}$ and indices $1\leq r(H)\leq m(H)$ such that
$\rho_{H,r(H),m(H)}^{(q)}\to0$ along the same sequence $H\to\infty$. Then 
we have convergence of the $p$-th absolute moment, that is 
\begin{equation}\label{eq:moment-convergence}
 \E\left[\left|\frac{\cD_H-M_{m(H)}}{\sqrt{V_{m(H)}}}\right|^p\right]
 \longrightarrow \E[|G|^p],
 \qquad G\sim\Normal(0,1).
\end{equation}
If $p$ is a positive integer, then, under the same assumptions on $q$ and
$r=r(H)$, we get convergence of the $p$-th moment:
\begin{equation}\label{eq:ordinary-moment-convergence}
 \E\left[\left(\frac{\cD_H-M_{m(H)}}{\sqrt{V_{m(H)}}}\right)^p\right]
 \longrightarrow \E[G^p]
\end{equation}
along the same sequence $H\to\infty$
\end{theorem}

Although the assumptions \eqref{eq:main-conditions} may be difficult to check as stated, we also provide a simpler sufficient condition for the  central limit theorem to hold with an explicit rate. Recall \eqref{eq:SQDelta}. We have the following
\begin{corollary}\label{cor:old-condition}
Suppose that, for some $C,\eta>0$,
\begin{equation}\label{eq:old-condition}
 Q_m\leq C m^{-1/2-\eta}S_m^2
\end{equation}
for all sufficiently large $m$.  Then \Cref{thm:main,thm:moments} hold as $H\to\infty$
through all positive integers.  Moreover,
\begin{equation}\label{eq:old-rate}
 \dK\left(\frac{\cD_H-M_{m(H)}}{\sqrt{V_{m(H)}}},\Normal(0,1)\right)
 =O\left((\log m(H))^{1/2}m(H)^{-1/12-\eta/6}\right)
\end{equation}
as $H\to\infty$.
\end{corollary}

\begin{remark}\label{rem:constant-radix}
    For a constant radix $\mathbf{b}=(q,q,q,\ldots)$, we have  $S_m=mq^2$ and $Q_m=mq^4$ and hence  condition \eqref{eq:old-condition} holds with $C=1$ and $\eta=\frac{1}{2}$. Note that 
\begin{equation}\label{eq:constant-base-formulas}
 M_m=m\frac{q^2-1}{12q},\qquad
 V_m=m\frac{3q^4+10q^2-13}{720q^2}
       +\frac{1-q^{-m}}{12}
\end{equation}
with $m=m(H)$  and $q^{m-1}<H\le q^{m}$.
Hence in this case the centering satisfies $M_m=O_q(\log H)$, while the standard-deviation scale satisfies $\sqrt{V_m}=O_q(\sqrt{\log H})$. The rate of convergence provided by \Cref{cor:old-condition} is $O_q((\log H)^{-1/6}(\log\log H)^{1/2})$. When the radices are bounded or have controlled (polynomial) growth, we can actually improve this rate of convergence, yielding $O_q((\log H)^{-1/5})$ for the constant radix case. 
\end{remark}

The next
result, which includes both bounded and polynomially growing radices, is a
direct corollary of \Cref{thm:main,thm:moments}.

\begin{corollary}[Bounded and polynomial radices]\label{cor:regular-radices}
Suppose that, for some $\alpha\geq0$ and constants
$0<c_-\leq c_+<\infty$,
\begin{equation}\label{eq:regular-radix-bounds}
 c_-j^\alpha\leq b_j\leq c_+j^\alpha
 \qquad(j\geq1).
\end{equation}
Then
\Cref{thm:main,thm:moments} hold uniformly through all $H\to\infty$. Moreover
\begin{equation}\label{eq:regular-rate}
 \dK\left(\frac{\cD_H-M_{m(H)}}{\sqrt{V_{m(H)}}}\:,\:\Normal(0,1)\right)=O_{\alpha,c_-,c_+}(
m(H)^{-1/5})
\end{equation}
uniformly for $H\geq2$.
\end{corollary}

\subsection{Related results}
An earlier mixed-radix temporal central limit theorem was stated by Avdeeva,
Cellarosi and Sinai~\cite[Theorem~3.2]{ACS}, in reference to an unpublished manuscript by Avdeeva. The assumption therein is the same as  \eqref{eq:old-condition} in \Cref{cor:old-condition}, although 
the claimed limit theorem was only qualitative, with no explicit rate.
We therefore supply a stronger quantitative generalisations (\Cref{thm:main} and \Cref{cor:old-condition}) of the
unpublished result announced in \cite{ACS}.   
\Cref{thm:moments} also gives moment convergence under the additional
hypothesis stated there.

In the fixed-base setting, the predecessor of these results is the temporal central limit theorem of Levin and Merzbach~\cite{LevinMerzbach}. Their argument uses a digital decomposition followed by Bernstein blocking: after truncation to finite digit windows, separated block sums become independent, and an independent-array Berry--Esseen theorem gives a Kolmogorov-Smirnov rate \(O((\log_q H)^{-1/7})\); their proof in fact yields \(O_q((\log H)^{-1/6}(\log\log H)^{3/2})\). This differs from our reverse-martingale--coboundary decomposition and application of Haeusler's martingale Berry--Esseen theorem. For comparison, our analysis yields a rate \(O_q((\log H)^{-1/5})\) in constant
radix.
Although their formal results assume a constant radix, Levin and Merzbach
state in \cite[Remark~1.2]{LevinMerzbach} that the same proof extends to
varying radices satisfying \(\sum_{j\le m}b_j=O(m)\).  This assertion is in general
false without an additional Feller-type condition:
\Cref{ex:dominant-radix} gives a radix sequence satisfying this bound for
which the  centered and normalized discrepancy has a nontrivial
non-Gaussian subsequential limit. 
Levin and Merzbach also announce a treatment of arbitrary radix sequences in a forthcoming paper.  We have not located a published or preprint temporal central limit theorem realizing that announcement.  Levin subsequently studied the same mixed-radix odometers in a deterministic multidimensional discrepancy problem~\cite{LevinHaltonII}.  More precisely, for each of finitely many  spatial coordinate $i$, Levin chooses a sequence of mixed-radices
\(\mathbf b^{(i)}=(b_{i,1},b_{i,2},\ldots)\) and uses the corresponding odometer to form a generalized Halton sequence, also allowing digit permutations. The assumptions in \cite{LevinHaltonII} are that \(\prod_{j=1}^{m}b_{i,j}\leq C^m\), uniformly in \(i\) and \(m\), and that $\gcd(b_{i,k},b_{i',\ell})=1$
 whenever $i\neq i'$.
Levin proves an optimal-order lower bound for the worst-case local
discrepancy 
and does not
provide a temporal limit theorem.

Borda~\cite{Borda} proved a
quantitative fixed-base CLT and large-deviation estimates for the integrated
discrepancy $\int_0^1D_N(y)\,dy$, for $L^p$ discrepancy, and for certain
smooth ergodic sums.  
Drmota, Larcher and
Pillichshammer~\cite{DLP} obtained precise distribution results for van der
Corput and related digital sequences.

A particularly close methodological precedent in a non-mixing temporal
setting is Paquette and Son~\cite{PaquetteSon}.  For primitive substitution
systems they encode a uniformly sampled time by a finite prefix--suffix path,
use total-variation couplings to replace the bulk of the induced finite-path
law by a stationary Markov path law, and then apply a Markov-chain central
limit theorem.  
In our analyis, the reference law is instead product Haar measure.

As a broader methodological comparison outside the non-mixing setting,
Korepanov~\cite{Korepanov} couples Birkhoff-sum processes associated with
different initial measures so that they remain almost surely close, and uses
this comparison to transfer an almost sure invariance principle for
nonuniformly expanding systems.  There the coupling expresses memory loss of
the underlying expanding dynamics rather than temporal recoding of a
non-mixing system.

Central limit phenomena for the discrepancy of fixed-base radical-inverse sequences with digit permutations were studied by Wohlfarter~\cite{Wohlfarter}.  Cantor-base van der Corput and Halton sequences, constructed from mixed-radix expansions of the same type as \eqref{eq:vdc}, have been studied through classical discrepancy bounds by Haddley, Lertchoosakul and Nair~\cite{HLN} and by Bedna\v{r}\'ik, Lertchoosakul, Marques and Trojovsk\'y~\cite{BLMT}.  Those works do not address the joint temporal--endpoint distribution studied here.

For the fixed radix \(2\), Levin~\cite{LevinDigital} proved a related central limit theorem for the local discrepancy of higher-dimensional digital nets and, as a corollary, initial blocks of digital sequences, with randomness coming jointly from a binary digital shift and the box endpoint.  A recent temporal central limit theorem for a multidimensional product adding machine is also due to Levin~\cite{LevinMulti}.  In that result, each spatial coordinate \(i\) uses a fixed prime radix \(p_i\), with the primes \(p_1,\ldots,p_s\) distinct.  The theorem treats a fixed box satisfying a digit nondegeneracy condition
and Lebesgue-almost every initial point.  

For the fixed-base van der Corput sequence, Beretti~\cite{Beretti}
studies the mean-square discrepancy obtained by translating an interval
of fixed length around the circle.  He characterizes its growth in terms
of the base-\(p\) expansion of the interval length.

  The terminology ``temporal limit theorem''
follows Dolgopyat and Sarig~\cite{DolgopyatSarig}. See also Bromberg and
Ulcigrai~\cite{BrombergUlcigrai} for rotations.

\subsection{Structure of the paper}
Section~2 derives a digit-by-digit decomposition of the local discrepancy
and rewrites each contribution in a form suited to conditional-expectation
arguments.  Section~3 first compares the joint distribution of the first \(r\) digits
of a uniform random time with that of the first \(r\) digits of a
Haar-random clock string.
Then it introduces the ideal Haar-clock model and develops its martingale--coboundary decomposition, computes its mean,
covariance structure, and variance exactly, and proves the auxiliary
estimates needed for a limit theorem.  Section~4 uses this
structure to prove a quantitative central limit theorem and moment convergence
for the Haar model.  Section~5 couples finite clock strings to Haar strings,
establishes a transfer principle for general random-time laws, and applies it
to uniform random times. This yields the our main results: \Cref{thm:main} and \Cref{thm:moments}.  This section also derives Corollaries \ref{cor:old-condition} and \ref{cor:regular-radices}.  Section~6 tests the scope and
limitations of the theory through two irregular radix sequences:
an example that is not covered by either corollary but for
which the main theorem yields a quantitative Gaussian limit, and a non-example  with a non-Gaussian subsequential
limit.  Section~7 records extensions left for future work.

\section{Digital decomposition}

Recalling the digit convention following
\eqref{B-expansion-real-x}, write every \(y\in[0,1)\) uniquely as
\begin{equation}\label{eq:y-digits}
 y=\sum_{i\geq0}\frac{a_i}{B_{i+1}},
 \qquad 0\leq a_i<b_{i+1}.
\end{equation}
Call \(q=(q_i)_{i\geq0}\) an \emph{admissible} digit string if
\(0\leq q_i<b_{i+1}\) for every \(i\). An admissible string need not
have finite support: the digits \(a=(a_i)_{i\geq0}\) of the endpoint
\(y\) of the interval appearing in \eqref{eq:D-N} may continue
indefinitely. By contrast, the digit string
\begin{equation}\label{eq:v71-auto-0002}
 c(n)=(c_i(n))_{i\geq0}
\end{equation}
of an integer \(n\), as in \eqref{digits_of_integers}, has finite
support, and we extend it by zeros.
For any admissible digit string \(q\) and \(i\geq0\), define
\begin{equation}\label{eq:tail-U}
 U_i(q):=\sum_{k=0}^{\infty}
 \frac{q_{i+k}}{b_{i+1}\cdots b_{i+k+1}}.
\end{equation}
Thus \(U_i(q)\) is the normalized mixed-radix tail beginning at
position \(i\), and
\begin{equation}\label{Ui_and_Ui+1}
 U_i(q)=\frac{q_i+U_{i+1}(q)}{b_{i+1}}.
\end{equation}
For \(0\leq d<b_{i+1}\), define
\begin{equation}\label{eq:tail-replaced}
 W_i(d,q):=
 \frac{d+U_{i+1}(q)}{b_{i+1}}=
 \frac{d}{b_{i+1}}
 +\frac{q_{i+1}}{b_{i+1}b_{i+2}}
 +\frac{q_{i+2}}{b_{i+1}b_{i+2}b_{i+3}}
 +\cdots.
\end{equation}
Thus \(W_i(d,q)\) is the normalized tail beginning at position \(i\)
obtained by replacing its initial digit \(q_i\) with \(d\), while
leaving the subsequent digits
\(q_{i+1},q_{i+2},\ldots\) unchanged.

\begin{lemma}[Digital decomposition of the discrepancy]
\label{lem:digital-decomposition}
Let \(y\) be as in \eqref{eq:y-digits}. Let \(N\geq1\) be an integer
and write
\begin{equation}
 e_i:=c_i(N),
 \qquad
 N=\sum_{i\geq0}e_iB_i.\label{e_is_Digital_Decomposition_Lemma}
\end{equation}
Then
\begin{equation}\label{eq:digital-decomposition}
 D_N(y)=\sum_{i\geq0}\xi_i(e,y),
\end{equation}
where the sum is finite and
\begin{equation}\label{eq:xi-original}
 \xi_i(e,y):=\sum_{d=0}^{e_i-1}
 \left[-U_i(a)+\1_{\{W_i(d,e)<U_i(a)\}}\right].
\end{equation}
\end{lemma}

\begin{proof}
Recall the definition \eqref{eq:D-N} of the local discrepancy. 
We  partition the set \(\{0,\ldots,N-1\}\) into blocks and
calculate the contribution of each block directly to 
\eqref{eq:D-N}.
For \(i\geq0\) and \(0\leq d<e_i\), define
\begin{equation}\label{eq:block-I-id}
 \mathcal I_{i,d}
 :=
 \left\{
 \sum_{j>i}e_jB_j+dB_i+s:
 0\leq s<B_i
 \right\}.
\end{equation}
By \eqref{digits_of_integers}, $\mathcal I_{i,d}$ consists precisely
of the integers whose digits satisfy $c_j(n)=e_j$ for $j>i$ and
$c_i(n)=d$, with arbitrary admissible digits below $i$. Therefore, organizing $n<N$ by the largest position where its digits differ
from those of $N$ gives the partition
\begin{equation}\label{eq:block-partition}
 \{0,\ldots,N-1\}
 =
 \bigsqcup_{i\geq0}\bigsqcup_{d=0}^{e_i-1}\mathcal I_{i,d}.
\end{equation}
Using this partition in the defining formula
\eqref{eq:D-N}, we obtain
\begin{equation}\label{eq:discrepancy-block-partition}
 D_N(y)
 =
 \sum_{i\geq0}\sum_{d=0}^{e_i-1}
 \sum_{n\in\mathcal I_{i,d}}
 \left(\1_{[0,y)}(x_n)-y\right).
\end{equation}
Now fix \(i\geq0\) and \(0\leq d<e_i\). By
\eqref{eq:block-I-id}, every \(n\in\mathcal I_{i,d}\) has a unique
representation
\begin{equation}
 n=\sum_{j>i}e_jB_j+dB_i+s,\label{eq:n=sum+dBi+s}
 \qquad 0\leq s<B_i.
\end{equation}
Write the base-$\mathbf{b}$ expansion of \(s\) as
 $s=\sum_{j=0}^{i-1}c_j(s)B_j$.
Since \(s<B_i\), one has \(c_j(s)=0\) for every \(j\geq i\). There are
no carries \eqref{eq:n=sum+dBi+s}, and hence
\begin{equation}\label{eq:v71-auto-0007}
 c_j(n)=
 \begin{cases}
  c_j(s),&0\leq j<i,\\
  d,&j=i,\\
  e_j,&j>i.
 \end{cases}
\end{equation}
It follows from the definition of \(x_n\) and \eqref{eq:tail-replaced} that
\begin{align}
 x_n
 &=
 \sum_{j=0}^{i-1}\frac{c_j(s)}{B_{j+1}}
 +\frac{d}{B_{i+1}}
 +\sum_{j>i}\frac{e_j}{B_{j+1}} =
 \sum_{j=0}^{i-1}\frac{c_j(s)}{B_{j+1}}
 +\frac{W_i(d,e)}{B_i}.
 \label{eq:block-point-form}
\end{align}

For $i\geq 1$ and  \(0\leq s<B_i\), define the integer
\begin{equation}\label{def:rho_i(s)}
 \rho_i(s)
 :=
 B_i\sum_{j=0}^{i-1}\frac{c_j(s)}{B_{j+1}}=c_{i-1}(s)
 +b_i c_{i-2}(s)
 +b_i b_{i-1}c_{i-3}(s)
 +\cdots
 +(b_i\cdots b_2)c_0(s).
 \end{equation}
For \(i=0\), we set \(\rho_0(0)=0\).
Note that $\rho_i$ is a permutation of $\{0,1,\ldots,B_i-1\}$. Indeed, as \(s\) ranges over
\(\{0,1,\ldots,B_i-1\}\), the digit vector
 $\bigl(c_0(s),\ldots,c_{i-1}(s)\bigr)$ from \eqref{def:rho_i(s)}
ranges once through all admissible digit vectors, i.e. through the set $\prod_{\ell=0}^{i-1}\{0,1,\ldots,b_{\ell+1}-1\}$. The latter is precisely the set of all admissible digit vectors of elements of $\{0,1,\ldots, B_i-1\}$. 

Now, using \eqref{def:rho_i(s)}, the expression  \eqref{eq:block-point-form} for $x_n$ becomes
\begin{equation}\label{eq:v71-auto-0008}
 x_n=\frac{\rho_i(s)+W_i(d,e)}{B_i}.
\end{equation}
Consequently, as \(n\) ranges through \(\mathcal I_{i,d}\), the points
\(x_n\) are precisely
\begin{equation}\label{eq:block-points}
 \frac{k+W_i(d,e)}{B_i},
 \qquad 0\leq k<B_i,
\end{equation}
each occurring exactly once.

Next, we write the endpoint \(y\) in a form similar  to \eqref{eq:block-point-form}. From
\eqref{eq:y-digits}, we have
 $y
 =
 \sum_{j=0}^{i-1}\frac{a_j}{B_{j+1}}
 +\sum_{j\geq i}\frac{a_j}{B_{j+1}}$ and multiplication by \(B_i\) gives
\begin{equation}\label{eq:y-prefix-tail}
 B_i y
 =
 B_i\sum_{j=0}^{i-1}\frac{a_j}{B_{j+1}}
 +U_i(a).
\end{equation}
The first term on the right-hand side is an integer. Moreover, the
digit convention following \eqref{B-expansion-real-x} gives
 $0\leq U_i(a)<1$.
It follows that
\begin{equation}\label{eq:floor-tail}
 \lfloor B_i y\rfloor
 =
 B_i\sum_{j=0}^{i-1}\frac{a_j}{B_{j+1}},
 \qquad
 B_i y=\lfloor B_i y\rfloor+U_i(a).
\end{equation}
Set
 $\ell_i:=\lfloor B_i y\rfloor\in\{0,1,\ldots, B_i-1\}$. 
By \eqref{eq:block-points} and \eqref{eq:floor-tail},
\begin{align}
 \#\{n\in\mathcal I_{i,d}:x_n<y\}
 &=
 \#\left\{
 0\leq k<B_i:
 \frac{k+W_i(d,e)}{B_i}
 <
 \frac{\ell_i+U_i(a)}{B_i}
 \right\}\label{eq:v71-auto-0188}\\
 &=
 \#\left\{
 0\leq k<B_i:
 k+W_i(d,e)<\ell_i+U_i(a)
 \right\}.\label{card_n_in_I_i,d_st_xn<y}
\end{align}
Since the sequence $e$ has finite support, we have that $0\leq W_i(d,e)<1$. Thus, if \(k<\ell_i\), then
\begin{equation}\label{eq:v71-auto-0009}
 k+W_i(d,e)<k+1\leq\ell_i\leq\ell_i+U_i(a),
\end{equation}
so the inequality in \eqref{card_n_in_I_i,d_st_xn<y} holds. If \(k>\ell_i\), then
\begin{equation}\label{eq:v71-auto-0010}
 k+W_i(d,e)\geq k\geq\ell_i+1
 >\ell_i+U_i(a)
\end{equation}
and the inequality in \eqref{card_n_in_I_i,d_st_xn<y} fails. Finally, for \(k=\ell_i\), the inequality
holds precisely when
\begin{equation}\label{eq:v71-auto-0011}
 W_i(d,e)<U_i(a).
\end{equation}
Therefore,
\begin{equation}\label{eq:block-count}
 \#\{n\in\mathcal I_{i,d}:x_n<y\}
 =
 \lfloor B_i y\rfloor
 +\1_{\{W_i(d,e)<U_i(a)\}}.
\end{equation}

Recall that the block \(\mathcal I_{i,d}\) contains exactly \(B_i\) integers. Using \eqref{eq:block-count} and \eqref{eq:y-prefix-tail}, we see that the block's contribution to 
\eqref{eq:discrepancy-block-partition} is 
\begin{align}
 \sum_{n\in\mathcal I_{i,d}}
 \left(\1_{[0,y)}(x_n)-y\right)
 &=
 \#\{n\in\mathcal I_{i,d}:x_n<y\}-B_i y\label{eq:v71-auto-0189}\\
 &=
 -U_i(a)+\1_{\{W_i(d,e)<U_i(a)\}}.\label{inner-sum-worked-out}
\end{align}
Therefore, \eqref{eq:block-partition} and \eqref{inner-sum-worked-out} give \eqref{eq:digital-decomposition} with $\xi_{i}(e,y)$ as in \eqref{eq:xi-original}. 

We are left to show that the sum in \eqref{eq:digital-decomposition} is finite. To this end, note the digit string \(e=c(N)\) has finite support. Hence
\(e_i=0\) for all sufficiently large \(i\), and the corresponding
inner sums are empty. 
\end{proof}

While \Cref{lem:digital-decomposition} identifies the discrepancy of interest as a sum of the $\xi_i(e,y)$'s, the next lemma allows us to rewrite each summand in a way that is more amenable to our probabilistic analysis.
Recall the definition of $U_i$ from \eqref{eq:tail-U}.


\begin{lemma}[Local normal form]\label{lem:local-normal-form} Let $y$ and $a=(a_i)_{i\geq0}$ be as in \eqref{eq:y-digits}, and $e=(e_i)_{i\geq0}$ as in \eqref{e_is_Digital_Decomposition_Lemma}.
Then, for every $i\geq0$, we have
\begin{equation}\label{eq:local-normal-form}
 \xi_i(e,y)=\min(a_i,e_i)
 +\1_{\{a_i<e_i\}}\1_{\{U_{i+1}(e)<U_{i+1}(a)\}}
 -\frac{e_i(a_i+U_{i+1}(a))}{b_{i+1}}.
\end{equation}
In particular,
\begin{equation}\label{eq:xi-bound}
 |\xi_i(e,y)|\leq e_i\leq b_{i+1}-1.
\end{equation}
\end{lemma}

\begin{proof} Since \eqref{eq:local-normal-form} is a statement for fixed $i$, let us use the shorthands \begin{equation}\label{eq:v71-auto-0268}
A:=a_i,\quad E:=e_i,\quad  B:=b_{i+1},\quad U:=U_{i+1}(a),\quad V:=U_{i+1}(e).
\end{equation}
By \eqref{Ui_and_Ui+1} and \eqref{eq:tail-replaced}, for $0\le d<E$, the indicator in
\eqref{eq:xi-original} is equal to 
$\1_{\{W_i(d,e)<U_i(a)\}}
 =\1_{\{d+V<A+U\}}.
$
Given $A, U, V$, we now count how many $d\in\{0,\ldots,E-1\}$ satisfy the inequality $d+V<A+U$.  
$ $
Every $d<A$ satisfies it, because
$d+V<d+1\leq A\leq A+U$, whereas every $d>A$ fails it, because
$d+V\geq A+1>A+U$.  The only remaining possibility is $d=A$.  This digit
belongs to the range of summation precisely when $A<E$, and in that case
the inequality reduces to $V<U$.  Thus
\begin{equation}\label{eq:local-indicator-count}
 \sum_{d=0}^{E-1}\1_{\{d+V<A+U\}}
 =\min(A,E)+\1_{\{A<E\}}\1_{\{V<U\}}.
\end{equation}
This argument also covers the cases $A=0$ or $E=0$; when $E=0$, both sides of
\eqref{eq:local-indicator-count} vanish.
There are exactly $E$ summands in \eqref{eq:xi-original}, and
$U_i(a)=(A+U)/B$.  Combining these facts with \eqref{eq:local-indicator-count}
gives
\eqref{eq:local-normal-form}.  Finally, since each summand in
\eqref{eq:xi-original} 
has absolute value at most one,
we get $|\xi_i(e,y)|\leq E$. By definition, $E\le B-1$ and we get the bound \eqref{eq:xi-bound}.
\end{proof}
We illustrate the previous two lemmata with a fully worked out example.
\begin{example}
\label{ex:digital-decomposition}
Let $\mathbf{b}=(2^2,3^2,5^2,7^2,\ldots)$ be the primes squared. Then
\begin{equation}\label{eq:example-dd-products}
 (B_0,B_1,B_2,B_3,B_4,\ldots)=(1,4,36,900,44100,\ldots).
\end{equation}
For $N=1982$, the digits $e_i=c_i(1982)$ satisfy
\begin{equation}\label{eq:example-dd-digits}
 1982=
 \underbrace{\textcolor{RubineRed}{2}}_{e_0}B_0
 +\underbrace{\textcolor{RedOrange}{0}}_{e_1}B_1
 +\underbrace{\textcolor{Green}{5}}_{e_2}B_2
 +\underbrace{\textcolor{blue}{2}}_{e_3}B_3,
\end{equation}
and $e_i=0$  for $i\geq4$.
In particular,
\begin{equation}\label{eq:example-dd-orbit-point}
 x_{1982}
 =\frac{\textcolor{RubineRed}{2}}{B_1}
  +\frac{\textcolor{RedOrange}{0}}{B_2}
  +\frac{\textcolor{Green}{5}}{B_3}
  +\frac{\textcolor{blue}{2}}{B_4}
 =\frac{22297}{44100}.
\end{equation}
The blocks from \eqref{eq:block-I-id} partition
$\{0,\ldots,1981\}$ as follows:
\begin{equation}\label{eq:example-dd-blocks}
\begin{array}{c|c|l}
 i & d & \mathcal I_{i,d} \\ \hline
 3 & 0,1
   & \{900d,1+900d,\ldots,899+900d\}
     \quad\text{(length $900$)} \\
 2 & 0,\ldots,4
   & \{1800+36d,1801+36d,\ldots,1835+36d\}
     \quad\text{(length $36$)} \\
 0 & 0,1
   & \{1980+d\}
     \quad\text{(length $1$)}
\end{array}
\end{equation}
There are no blocks at level $i=1$, since
$e_1=\textcolor{RedOrange}{0}$, or at levels $i\geq4$.
Consider the endpoint with digits $a_i=1$ for every $i\geq0$, that is
\begin{equation}\label{eq:example-dd-endpoint}
 y=\sum_{j\geq1}\frac{1}{B_j}
 \approx0.278911753141375.
\end{equation}

The replaced tails \eqref{eq:tail-replaced} are
\begin{align}
 W_0(d,e)
 &=\frac{d}{4}
   +\frac{\textcolor{RedOrange}{0}}{4
\cdot 9}
   +\frac{\textcolor{Green}{5}}{4\cdot 9\cdot 25}
   +\frac{\textcolor{blue}{2}}{4\cdot 9\cdot 25\cdot 49},
   \qquad 0\leq d<4,
   \label{eq:example-dd-replaced-tail}\\
 W_1(d,e)
 &=\frac{d}{9}
   +\frac{\textcolor{Green}{5}}{9\cdot25}
   +\frac{\textcolor{blue}{2}}{9\cdot25\cdot49},
   \qquad 0\leq d<9,
   \label{eq:example-dd-replaced-tail-1}\\
 W_2(d,e)
 &=\frac{d}{25}
   +\frac{\textcolor{blue}{2}}{25\cdot49},
   \qquad 0\leq d<25,
   \label{eq:example-dd-replaced-tail-2}\\
 W_3(d,e)
 &=\frac{d}{49},
   \qquad 0\leq d<49,
   \label{eq:example-dd-replaced-tail-3}\\
 W_i(d,e)
 &=\frac{d}{b_{i+1}},
   \qquad 0\leq d<b_{i+1}\quad(i\geq4).
   \label{eq:example-dd-replaced-tail-higher}
\end{align}

Although these tails are defined at every level, the sums in
\eqref{eq:xi-original} are empty for $i=1$ and for $i\geq4$,
since the corresponding digits $e_i$ vanish.
The tail comparisons in \eqref{eq:local-normal-form} are
\begin{align}
 U_1(e)
 &=\frac{\textcolor{RedOrange}{0}}{9}
   +\frac{\textcolor{Green}{5}}{9\cdot25}
   +\frac{\textcolor{blue}{2}}{9\cdot25\cdot49}
   =\frac{247}{11025}<\frac19<U_1(a),
   \label{eq:example-dd-tail-1}\\
 U_3(e)
 &=\frac{\textcolor{blue}{2}}{49}
   >\frac{1+U_4(a)}{49}=U_3(a),
   \label{eq:example-dd-tail-3}\\
 U_4(e)&=0<U_4(a).
   \label{eq:example-dd-tail-4}
\end{align}
Here $0<U_j(a)<1$ for every $j$. Thus, among the digits
$0\leq d<e_i$, the inequality $W_i(d,e)<U_i(a)$ holds precisely for
$d=0,1$ at level $i=0$, for $d=0$ at level $i=2$, and for
$d=0,1$ at level $i=3$. Equivalently, the additional indicator term in
\eqref{eq:local-normal-form} equals $1$ at levels $i=0,3$ and $0$
at level $i=2$. Consequently,
\begin{align}
 \xi_0(e,y)
 &=2-\textcolor{RubineRed}{2}y
   \approx1.4421764937,
   \label{eq:example-dd-xi-0}\\
 \xi_1(e,y)&=\textcolor{RedOrange}{0},
   \label{eq:example-dd-xi-1}\\
 \xi_2(e,y)
 &=1-\textcolor{Green}{5}U_2(a)
   \approx0.7958844346,
   \label{eq:example-dd-xi-2}\\
 \xi_3(e,y)
 &=2-\textcolor{blue}{2}U_3(a)
   \approx1.9588443455,
   \label{eq:example-dd-xi-3}
\end{align}
and $\xi_i(e,y)=0$ for $i\geq4$.
Since $U_2(a)=36y-10$ and $U_3(a)=900y-251$,
\Cref{lem:digital-decomposition} gives
\begin{equation}\label{eq:example-dd-discrepancy}
 D_{1982}(y)=\sum_{i=0}^{3}\xi_i(e,y)
 =557-1982y\approx4.1969052738.
\end{equation}
Thus exactly $557$ of the points $x_0,\ldots,x_{1981}$ lie in $[0,y)$.
\end{example}

\section{The Haar digit model and an exact martingale decomposition}
Write $m_G$ for Haar probability measure on the compact abelian group $G$.
Recalling \eqref{eq:y-digits} and the cutting-and-stacking construction of
\Cref{sec:intro}, the digit map
\begin{equation}\label{eq:v71-auto-0012}
 [0,1)\ni y\longmapsto a(y)=(a_0(y),a_1(y),a_2(y),\ldots)\in G
\end{equation}
sends Lebesgue measure to $m_G$.  Indeed, prescribing the first $r$ digits
selects a half-open interval of length $B_r^{-1}$.  Consequently, for the
uniform endpoint $Y$ in \eqref{eq:random-statistic}, the digits
$a_i(Y)$ are independent and
\begin{equation}\label{eq:v71-auto-0013}
 a_i(Y)\sim\operatorname{Unif}\{0,\ldots,b_{i+1}-1\}.
\end{equation}
Thus the Haar law of the random endpoint $Y$ digit string is exact.

By contrast, for each fixed horizon $H\geq1$, the random time $N_H=\lfloor HZ\rfloor$ is
uniform on $\{0,\ldots,H-1\}$ and its clock string $c(N_H)$ has finite support.
When $H\geq2$, \eqref{H-sandwiched} shows more
precisely that its digits vanish at every index $i\geq m(H)$.  The next lemma makes precise the sense
in which the law of the clock string $c(N_H)$ nevertheless approaches the
Haar measure $m_G$ as $H\to\infty$. Below, $\TV$ denotes the total variation
distance. Recall the definition \eqref{eq:epsilon-time} of $\eps_{H,r}$.

\begin{lemma}[Haar approximation of random-time digits]
\label{lem:time-digits-Haar}
Let $H,r\geq1$, and write $H=kB_r+s$ with $0\leq s<B_r$.  Then
\begin{equation}\label{eq:time-prefix-TV}
 \TV\left(
  \mathcal L\bigl(c_0(N_H),\ldots,c_{r-1}(N_H)\bigr),
  \bigotimes_{i=0}^{r-1}
  \operatorname{Unif}\{0,\ldots,b_{i+1}-1\}
 \right)
 =\eps_{H,r}. 
\end{equation}
In particular,
\begin{equation}\label{eq:time-digits-weak-Haar}
 \mathcal L\bigl(c(N_H)\bigr)\Longrightarrow m_G
 \qquad \mbox{as}\:\: H\to\infty
\end{equation}
weakly on $G$.  
\end{lemma}

\begin{proof}
The first $r$ digits of $N_H$ are in bijection with the residue modulo $B_r$, since
\begin{equation}\label{eq:v71-auto-0014}
 N_H\bmod B_r=\sum_{i=0}^{r-1}c_i(N_H)B_i.
\end{equation}
Moreover, the map $(d_0,\ldots,d_{r-1})\mapsto 
\sum_{i=0}^{r-1}d_iB_i$ is a bijection from
 $\prod_{i=0}^{r-1}\{0,\ldots,b_{i+1}-1\}$
  onto
 $\{0,\ldots,B_r-1\}$, mapping the product measure $\bigotimes_{i=0}^{r-1}
  \operatorname{Unif}\{0,\ldots,b_{i+1}-1\}$ to the uniform measure on $\{0,1,\ldots, B_r-1\}$.
 Hence  we have
\begin{equation}\label{eq:v71-auto-0015}
\begin{aligned}
 &\TV\left(
  \mathcal L\bigl(c_0(N_H),\ldots,c_{r-1}(N_H)\bigr),
  \bigotimes_{i=0}^{r-1}
  \operatorname{Unif}\{0,\ldots,b_{i+1}-1\}
 \right)
\\
 &\qquad =
 \TV\left(
  \mathcal L(N_H\bmod B_r),
  \operatorname{Unif}\{0,\ldots,B_r-1\}
 \right).
\end{aligned}
\end{equation}
Since $H=kB_r+s$ with $0\leq s<B_r$, among the integers
$0,\ldots,H-1$, exactly $s$ residue classes modulo $B_r$ occur $k+1$
times and the remaining $B_r-s$ classes occur $k$ times.  We have
 $\frac{k+1}{H}-\frac1{B_r}=\frac{B_r-s}{HB_r}$ and 
 $\frac1{B_r}-\frac kH=\frac{s}{HB_r}$.
Summing the positive deviations from the uniform law gives a total variation distance of 
\begin{equation}\label{eq:v71-auto-0016}
 s\left(\frac{k+1}{H}-\frac1{B_r}\right)
 =\frac{s(B_r-s)}{HB_r},
\end{equation}
which proves \eqref{eq:time-prefix-TV} after recalling \eqref{eq:epsilon-time}.
For each fixed $r$, this distance is at most $B_r/H$ and hence tends to
zero as $H\to\infty$.  Finite-digit cylinder sets form a convergence-determining class for the product topology on $G$, proving
\eqref{eq:time-digits-weak-Haar}.
\end{proof}

\begin{example}
\label{ex:time-digits-1982}
Return to the radices of \Cref{ex:digital-decomposition}, now taking
$H=1982$. Thus $N_H$ is uniform on $\{0,\ldots,1981\}$, and the
blocks in \eqref{eq:example-dd-blocks} partition its possible values.
The fixed digit string $(2,0,5,2,0,\ldots)$ is now $c(H)$;
the random clock string is $c(N_H)$.

The block partition makes the distribution of the first two digits
particularly transparent. Modulo $B_2=36$, each of the two blocks
of length $900$ contains every residue exactly $25$ times,
and each of the five blocks of length $36$ contains every residue
once. The two singletons, $1980$ and $1981$, have residues $0$ and $1$.
Consequently,
\begin{equation}\label{eq:example-time-residues-36}
 \Pp(N_H\:(\bmod\:36)=j)=
 \begin{cases}
  \tfrac{56}{1982},&j=0,1,\\[4pt]
  \tfrac{55}{1982},&j=2,\ldots,35.
 \end{cases}
\end{equation}
Since each remainder modulo $36$ corresponds uniquely to a pair
of digits through
\begin{equation}\label{eq:example-time-residue-digits}
 N_H
 (\bmod\:36)=c_0(N_H)+4c_1(N_H),
\end{equation}
the probabilities in \eqref{eq:example-time-residues-36} also
specify the joint law of $(c_0(N_H),c_1(N_H))$.

More generally, writing $H=kB_r+s$ with $0\leq s<B_r$,
\Cref{lem:time-digits-Haar} gives the following exact total variation
distances between the first $r$ digits and their product Haar law:
\begin{equation}\label{eq:example-time-TV-table}
\begin{array}{c|r|r|r|c}
 r & B_r & k & s & \eps_{1982,r}\\ \hline
 1 & 4 & 495 & 2
   & \tfrac{1}{1982}\approx0.000505\\[6pt]
 2 & 36 & 55 & 2
   & \tfrac{17}{17838}\approx0.000953\\[6pt]
 3 & 900 & 2 & 182
   & \tfrac{32669}{445950}\approx0.073257\\[6pt]
 4 & 44100 & 0 & 1982
   & \tfrac{21059}{22050}\approx0.955057
\end{array}
\end{equation}
Thus the first two digits are jointly close to independent uniform
digits, while including the fourth digit gives a large total
variation distance. This illustrates why the later coupling argument
approximates a suitable number of initial digits and controls the
remaining levels separately. Note that 
on the full digit space, the distance is $\TV\bigl(\mathcal L(c(N_H)),m_G\bigr)=1$. 
Indeed, the law of $c(N_H)$ is supported on the finite set
$\{c(0),\ldots,c(1981)\}$, which has Haar measure zero.
\end{example}

\begin{remark}[Maximal coupling and the failure probability]
\label{rem:maximal-coupling}
Let $\mu$ and $\nu$ be probability measures on a finite set $S$.  A
\emph{coupling} of $\mu$ and $\nu$ is a pair of random variables $(X,Y)$
defined on a common probability space with $\mathcal L(X)=\mu$ and
$\mathcal L(Y)=\nu$.  The total variation distance can be characterised via
couplings as\footnote{Indeed,
$\TV(\mu,\nu)
 =\frac12\sum_{x\in S}|\mu(x)-\nu(x)|
 =1-\sum_{x\in S}\min\{\mu(x),\nu(x)\}$.
For every coupling, we have
\begin{equation}\label{eq:maximal-coupling-agreement-bound}
 \Pp(X=Y)
 =\sum_{x\in S}\Pp(X=Y=x)
 \leq\sum_{x\in S}\min\{\mu(x),\nu(x)\}.
\end{equation}}
\begin{equation}\label{eq:TV-coupling-characterization}
 \TV(\mu,\nu)
 =\inf_{\substack{\mathcal L(X)=\mu\\ \mathcal L(Y)=\nu}}
   \Pp(X\ne Y),
 \qquad
 1-\TV(\mu,\nu)
 =\sup_{\substack{\mathcal L(X)=\mu\\ \mathcal L(Y)=\nu}}
   \Pp(X=Y).
\end{equation}
A coupling attaining either, and hence both, of the extrema in
\eqref{eq:TV-coupling-characterization} is called a \emph{maximal coupling}.
A construction of such a coupling is useful below, so we make one explicit.

For $x\in S$, put
\begin{equation}\label{eq:maximal-coupling-residual-parts}
 c_x:=\min\{\mu(x),\nu(x)\},\qquad
 \mu^\circ(x):=\mu(x)-c_x,\qquad
 \nu^\circ(x):=\nu(x)-c_x,
\end{equation}
and let
\begin{equation}\label{eq:maximal-coupling-delta}
 \delta:=1-\sum_{x\in S}c_x=\TV(\mu,\nu).
\end{equation}
The two residual measures have the same total mass,
\begin{equation}\label{eq:maximal-coupling-residual-masses}
 \sum_{x\in S}\mu^\circ(x)
 =\sum_{x\in S}\nu^\circ(x)=\delta,
\end{equation}
and their supports are disjoint, since
\begin{equation}\label{eq:maximal-coupling-disjoint-support}
 \mu^\circ(x)\nu^\circ(x)=0
 \qquad\text{for every }x\in S.
\end{equation}
Thus, after placing mass $c_x$ at each diagonal point $(x,x)$, one must
match the remaining first-marginal mass $\mu^\circ$ with the remaining
second-marginal mass $\nu^\circ$ in order to complete the prescribed
marginals.  If $\delta>0$, one possible matching gives $(X,Y)$ the joint law
\begin{equation}\label{eq:maximal-coupling-product-law}
 \pi(x,y):=c_x\1_{\{x=y\}}
 +\frac{\mu^\circ(x)\nu^\circ(y)}{\delta}.
\end{equation}
If $\delta=0$, instead take $X=Y$ with common law $\mu=\nu$.  When
$\delta>0$, the second term in \eqref{eq:maximal-coupling-product-law}
gives only one possible residual matching.  More generally, it may be
replaced by any joint measure whose first residual marginal is
$\mu^\circ$ and whose second residual marginal is $\nu^\circ$.  Such a
matching is generally not unique; nevertheless, disjointness of the residual
supports forces every choice off the diagonal, and hence every choice gives
\begin{equation}\label{eq:maximal-coupling-success-failure}
 \Pp(X=Y)=\sum_{x\in S}c_x=1-\delta,
 \qquad
 \Pp(X\ne Y)=\delta.
\end{equation}
Thus every such residual matching produces a maximal coupling.

To ``maximally couple'' two random variables means to coordinate their joint
law so that their probability of equality is as large as possible; it does
not mean to make them independent.  

We apply the maximal-coupling construction
\eqref{eq:maximal-coupling-residual-parts}--\eqref{eq:maximal-coupling-product-law}, with any choice of residual
matching, to
\begin{equation}\label{eq:maximal-coupling-time-residues}
 R=N_H\bmod B_r,
 \qquad
 R^*\sim\operatorname{Unif}\{0,\ldots,B_r-1\}.
\end{equation}
By \Cref{lem:time-digits-Haar}, a maximal coupling can be chosen so that
\begin{equation}\label{eq:time-maximal-coupling}
 \Pp(R\ne R^*)
 =\TV\left(\mathcal L(R),
       \operatorname{Unif}\{0,\ldots,B_r-1\}\right)
 =\eps_{H,r}.
\end{equation}

For a concrete illustration, take $H=1982$ and $r=2$, as in
\Cref{ex:time-digits-1982}. The law of $R=N_H\bmod36$ is given by
\eqref{eq:example-time-residues-36}: residues $0,1$ each have
probability $56/1982$, while every other residue has probability
$55/1982$.
We construct $R^*$ by keeping $R^*=R$ whenever
$R\in\{2,\ldots,35\}$. When $R\in\{0,1\}$, we keep $R^*=R$
with probability
\begin{equation}\label{eq:example-time-coupling-keep}
 \alpha:=\frac{1/36}{56/1982}=\frac{991}{1008},
\end{equation}
and otherwise choose $R^*$ uniformly from $\{2,\ldots,35\}$.
The probability of changing the residue is therefore
\begin{equation}\label{eq:example-time-coupling-failure}
 \Pp(R\ne R^*)
 =\frac{112}{1982}(1-\alpha)
 =\frac{17}{17838}
 =\eps_{1982,2}.
\end{equation}
To verify that $R^*$ is uniform, observe that $\Pp(R^*=0)=\Pp(R^*=1)=\frac1{36}$ by \eqref{eq:example-time-coupling-keep} and
\begin{align}
 \Pp(R^*=j)
 &=\frac{55}{1982}
   +\frac1{34}\frac{17}{17838}
   =\frac1{36},
   \qquad j=2,\ldots,35.
   \label{eq:example-time-coupling-mass-rest}
\end{align}
Thus $R^*$ has the required uniform law, and the mismatch
probability equals the total variation distance, so this coupling
is maximal. Equivalently, the first two random-time digits agree
simultaneously with their Haar counterparts with probability
$17821/17838\approx0.999047$.

In view of the residue--digit bijection \eqref{eq:v71-auto-0014}, the event
$\{R=R^*\}$ is exactly the event that the first $r$ digits of the random time
agree with the first $r$ Haar digits.  Therefore $\eps_{H,r}$ is the
probability that this maximal coupling fails to make those digit prefixes
agree.  It is also the least possible mismatch probability among all such
couplings; this is the sense in which we call it the \emph{failure
probability}.  For a closely related use of total-variation couplings in a
temporal central limit theorem for substitution dynamics, see Paquette and
Son~\cite[Sec.~5]{PaquetteSon}.
\end{remark}

More generally, if the depth $r=r(H)$ is allowed to grow with $H$, then
\eqref{eq:time-prefix-TV} shows that the first $r(H)$ digits approach
their Haar law in total variation precisely when $\eps_{H,r(H)}\to0$ as $H\to\infty$.
For example, $B_{r(H)}=o(H)$ is sufficient, while the distance is exactly
zero whenever $B_{r(H)}$ divides $H$.

Weak convergence in \eqref{eq:time-digits-weak-Haar} controls every fixed number of initial digits, whereas the
digital decomposition of $\cD_H$ may involve all levels
$0,\ldots,m(H)-1$, with $m(H)\to\infty$. 
We therefore first analyze an idealized model with a Haar clock
string.  The transfer to uniform random times, using
\eqref{eq:time-prefix-TV} with a suitable auxiliary depth
$r=r(H)\leq m(H)$, is carried out in
\Cref{lem:clock-coupling} and \Cref{prop:general-time}, and completed in the proof
of \Cref{thm:main}.

\subsection{The Haar model}
Let $a=(a_i)_{i\geq0}$ and
$e=(e_i)_{i\geq0}$ be two independent Haar-random elements of $G$. In particular, all digits are
independent and
\begin{equation}\label{eq:v71-auto-0017}
 a_i,e_i\sim\operatorname{Unif}\{0,\ldots,b_{i+1}-1\}.
\end{equation}
As discussed above, $a_i$ plays the exact role of $a_i(Y)$  with $Y\sim\operatorname{Unif}[0,1)$, while $e_i$ is only an idealization of $c_i(N_H)$ with $N_H\sim\operatorname{Unif}\{0,1,\ldots, H-1\}$ in the limiting sense of \eqref{eq:time-digits-weak-Haar}.
Recall \eqref{eq:tail-U}. Motivated by \eqref{eq:local-normal-form}, we introduce the random variables
\begin{equation}\label{eq:Xi-product-explicit}
 \Xi_i:=\min(a_i,e_i)
 +\1_{\{a_i<e_i\}}\1_{\{U_{i+1}(e)<U_{i+1}(a)\}}
 -\frac{e_i(a_i+U_{i+1}(a))}{b_{i+1}},\quad i\ge0.
\end{equation}
 Define
\begin{equation}\label{eq:X-psi-filtration}
 X_i:=\Xi_i-\E[\Xi_i],
 \qquad
 \psi_i:=\1_{\{U_i(e)<U_i(a)\}}-U_i(a),
 \qquad
 \cF_i:=\sigma(a_j,e_j:j\geq i).
\end{equation}
The next lemma gathers several explicit computations.
\begin{lemma}\label{lem:one-level}
For every $i\geq0$, 
\begin{align}
 \E[\Xi_i]&=\mu(b_{i+1}),\label{eq:one-mean}\\
 \Var(\Xi_i)&=
 \frac{3b_{i+1}^4+70b_{i+1}^2-60b_{i+1}-13}{720\,b_{i+1}^2},\label{eq:one-var}\\
 \E[X_i\mid\cF_{i+1}]
 &=\frac{b_{i+1}-1}{2\,b_{i+1}}\psi_{i+1},\label{eq:conditional-X}\\
 \E[\psi_i\mid\cF_{i+1}]
 &=\frac{1}{b_{i+1}}\psi_{i+1},\label{eq:conditional-psi}\\
 \E[\psi_iX_i]&=-\frac{b_{i+1}-1}{12\,b_{i+1}}.
 \label{eq:psi-X}\\
 \E[\psi_i^2]&=\frac16.
 \label{eq:psi-second-moment}
\end{align}
\end{lemma}

\begin{proof}
We again use the shorthands in \eqref{eq:v71-auto-0268}, together with
$J=\1_{\{V<U\}}$. Hence $\Xi_i=\min(A,E)+\1_{\{A<E\}}J-\frac{E(A+U)}B$.
The indicator in the definition of $\psi_i$ can be written using \eqref{Ui_and_Ui+1} as
\begin{equation}\label{eq:current-comparison-K}
 K:=\1_{\{U_i(e)<U_i(a)\}}
 =\1_{\{E+V<A+U\}}
 =\1_{\{E<A\}}+\1_{\{E=A\}}J.
\end{equation}
Indeed, if $E<A$ the inequality holds for every $U,V\in[0,1)$, if
$E>A$ it never holds, and if $E=A$ it holds  precisely when $V<U$.  
We use \eqref{eq:current-comparison-K} to rewrite
\begin{equation}\label{eq:psi-short}
 \psi_i=K-\frac{A+U}{B},
 \qquad
 \psi_{i+1}=J-U.
\end{equation}
Note that $U,V$ are
independent uniform variables and are independent of $A,E$.  In particular,
\begin{equation}\label{EU_EJ_EU^2_EUJ}
 \E[U]=\E[J]=\frac12,
 \qquad \E[U^2]=\frac13,
 \qquad \E[UJ]=\int_0^1u\Pp(V<u)\,du=\int_0^1u^2\,du=\frac13.
\end{equation}
Conditioning first on $A,E$ and using  $J^2=J$ along with \eqref{EU_EJ_EU^2_EUJ}, we get
\begin{align}
 \E[\Xi_i\mid A,E]
 &=\begin{cases}
 (1-E/B)(A+\tfrac12),&A<E,\\
 E\bigl(1-(A+\tfrac12)/B\bigr),&A\geq E,
 \end{cases}\label{eq:v71-auto-0190}\\
 \E[\Xi_i^2\mid A,E]
 &=\begin{cases}
 (1-E/B)^2(A^2+A)+\tfrac12-\dfrac{2E}{3B}
       +\dfrac{E^2}{3B^2},&A<E,\\
 E^2\left\{(1-A/B)^2-\dfrac1B(1-A/B)
       +\dfrac1{3B^2}\right\},&A\geq E.
 \end{cases}\label{eq:v71-auto-0191}
\end{align}
Indeed, on $A<E$ we have
$\Xi_i=A(1-E/B)+J-(E/B)U$, while on $A\geq E$ we have
$\Xi_i=E(1-A/B)-(E/B)U$.

For fixed $E=e$, summing the first conditional moment over $A$ yields
\begin{align}
 \sum_{A=0}^{B-1}\E[\Xi_i\mid A,E=e]
 &=\frac{B-e}{B}\sum_{A=0}^{e-1}(A+\tfrac12)
   +e\sum_{A=e}^{B-1}\left(1-\frac{A+1/2}{B}\right)=\frac{e(B-e)}2.\label{eq:v71-auto-0192}
\end{align}
Similarly, after putting $k=B-A$ in the second region,
\begin{align}
 &\sum_{A=0}^{B-1}\E[\Xi_i^2\mid A,E=e]\label{eq:v71-auto-0193}\\
 &=\frac{(B-e)^2}{B^2}\sum_{A=0}^{e-1}(A^2+A)
   +e\left(\frac12-\frac{2e}{3B}+\frac{e^2}{3B^2}\right)+\frac{e^2}{B^2}\sum_{k=1}^{B-e}
       \left(k^2-k+\frac13\right)
 \label{sum-conditional-second-moment-Xi}
\end{align}
Using the identities 
$\displaystyle\sum_{A=0}^{e-1}(A^2+A)=\tfrac{e(e^2-1)}{3}$ and
$\displaystyle\sum_{k=1}^{B-e}(k^2-k+\tfrac{1}{3})=\tfrac{(B-e)^3}{3}$, we simplify 
\eqref{sum-conditional-second-moment-Xi} to 
$\frac e6+\frac{e^2(B-e)^2}{3B}$.
Consequently, summing over $e$ and using standard power-sum identities, we get
\begin{align}
 \E[\Xi_i]
 &=\frac{1}{B^2}\sum_{e=0}^{B-1}\frac{e(B-e)}{2}
   =\frac{B^2-1}{12B},\label{eq:v71-auto-0194}\\
 \E[\Xi_i^2]
 &=\frac{1}{B^2}\sum_{e=0}^{B-1}\left(\frac e6
       +\frac{e^2(B-e)^2}{3B}\right)=\frac{(B-1)(2B^3+2B^2+17B+2)}{180 B^2}.\label{eq:v71-auto-0195}
\end{align}
These identities prove \eqref{eq:one-mean} and \eqref{eq:one-var}.

We now condition on $\cF_{i+1}$.  This sigma-algebra fixes $U,V$, and hence
$J$, while $A,E$ remain independent uniform digits.  
The elementary digit averages $\tfrac1{B^2}\sum_{A,E=0}^{B-1}\min(A,E)
 =\tfrac{(B-1)(2B-1)}{6B}$, $\tfrac1{B^2}\sum_{A,E=0}^{B-1}\1_{\{A<E\}}=\tfrac{B-1}{2B}$,  and $\frac1{B^2}\sum_{A,E=0}^{B-1}\tfrac{E(A+U)}B
 =\tfrac{(B-1)^2}{4B}+\tfrac{(B-1)U}{2B}$
give
\begin{equation}\label{eq:v71-auto-0019}
 \E[\Xi_i\mid\cF_{i+1}]
 =\mu(B)+\frac{B-1}{2B}(J-U)
\end{equation}
and using \eqref{eq:psi-short} and \eqref{eq:one-mean} we immediately get \eqref{eq:conditional-X}.
  Similarly, $\frac1{B^2}\sum_{A,E=0}^{B-1}K
 =\frac{B-1}{2B}+\frac JB$ and  $\frac1{B^2}\sum_{A,E=0}^{B-1}\frac{A+U}{B}
 =\frac{B-1}{2B}+\frac UB$ yield 
\begin{equation}\label{eq:v71-auto-0020}
 \E[\psi_i\mid\cF_{i+1}]
 =\frac{J-U}{B}
 =\frac1B\psi_{i+1},
\end{equation}
which is \eqref{eq:conditional-psi}.

Next we compute the mixed moment $\E[\psi_i X_i]$.  Condition on $A,U$ and put $w=(A+U)/B=U_i(a)$.  Then $U_i(e)$ is
independent and uniform on $[0,1)$, while $K=\1_{\{U_i(e)<w\}}$ and
$\psi_i=K-w$.  Hence $\E[\psi_i\mid A,U]=0$. Averaging first over
$E$ and $V$ gives
\begin{align}
 \E[\Xi_i\mid A,U]
 &=\frac1B\sum_{e=0}^{B-1}
   \left\{\min(A,e)+\1_{\{A<e\}}U-ew\right\}\label{eq:v71-auto-0196}\\
 &=\frac1B\left(\frac{A(A-1)}2+A(B-A)+(B-A-1)U
       -\frac{B(B-1)}2w\right)\label{eq:v71-auto-0197}\\
 &=\frac{A(B-A)+(B-2A-1)U}{2B}.\label{eq:v71-auto-0198}
\end{align}
On the event $\{K=1\}$ one has $\Xi_i=E(1-w)$.  Therefore
\begin{equation}\label{eq:v71-auto-0021}
 \E[K\Xi_i\mid A,U]
 =\frac{1-w}{B}\left(\sum_{e=0}^{A-1}e+AU\right)
 =\frac{A(A-1+2U)(B-A-U)}{2B^2}.
\end{equation}
Subtracting $w\E[\Xi_i\mid A,U]$ now yields
\begin{equation}\label{eq:v71-auto-0022}
 \E[\psi_iX_i\mid A,U]
 =\frac{-A(B-A)+2AU-(B-1)U^2}{2B^2}.
\end{equation}
Averaging over the independent uniform digit $A$ and uniform variable $U$,
and using
the identity
 $\sum_{A=0}^{B-1}A(B-A)=\frac{B(B-1)(B+1)}6$, we conclude that
\begin{align}
 \E[\psi_iX_i]
 &=\frac1{2B^3}\sum_{A=0}^{B-1}
   \left(-A(B-A)+A-\frac{B-1}{3}\right)=-\frac{B-1}{12B},\label{eq:v71-auto-0199}
\end{align}
which proves \eqref{eq:psi-X}.

Finally, conditional on $U_i(a)=u$, the indicator
$\1_{\{U_i(e)<U_i(a)\}}$ is Bernoulli with parameter $u$.  Hence
\begin{equation}\label{eq:v71-auto-0023}
 \E[\psi_i^2]
 =\int_0^1\left[u(1-u)^2+(1-u)u^2\right]du=\frac16,
\end{equation}
proving \eqref{eq:psi-second-moment}.
\end{proof}

Our analysis will provide a central limit theorem for the sequence of partial sums of the variable $\Xi_i$'s.
For $m\geq1$, let us define
\begin{equation}\label{eq:Pm-def}
 \mathcal P_m:=\sum_{i=0}^{m-1}\Xi_i.
\end{equation} 
It follows immediately from \eqref{eq:one-mean} and  \eqref{eq:M_and_T_and_V} that 
\begin{equation}\label{eq:product-exact-mean}
 \E\left[\mathcal P_m\right]=M_m.
\end{equation}

By \eqref{eq:conditional-X}, the centred random variables
$X_i=\Xi_i-\mu(b_{i+1})$ are not martingale differences; neither are the
$\Xi_i$. In the next proposition, we show that we can nevertheless decompose
their partial sums as a combination of a reverse\footnote{Note that
\eqref{eq:X-psi-filtration} gives $\cF_{i+1}\subseteq\cF_i$ and hence we have
a reverse filtration.}-martingale-difference sum and a coboundary.
Define
\begin{equation}\label{eq:d-def}
 d_i=X_i+\frac12(\psi_i-\psi_{i+1}),\quad i\geq0, 
\end{equation} so that 
\begin{equation}\label{eq:martingale-coboundary}
 \sum_{i=0}^{m-1}X_i
 =\sum_{i=0}^{m-1}d_i+\frac12(\psi_m-\psi_0),\quad m\geq1.
\end{equation}

\begin{proposition}[Martingale--coboundary decomposition]\label{prop:martingale-decomp}
The sequence $(d_i,\mathcal F_i)_{i\geq0}$ is a reverse
martingale-difference sequence. That is, for every $i\geq0$,
\begin{equation}\label{eq:reverse-MD}
 d_i\text{ is }\mathcal F_i\text{-measurable},\qquad
 \E[d_i^2]<\infty,\qquad
 \E[d_i\mid\mathcal F_{i+1}]=0.
\end{equation}
Moreover, with
\begin{equation}\label{eq:G-def}
 G_i:=U_i(a)^2+\1_{\{U_i(e)<U_i(a)\}}(1-2U_i(a)),
\end{equation}
we have
\begin{equation}\label{eq:conditional-d2}
 \E[d_i^2\mid\cF_{i+1}]
 =\frac{b_{i+1}^2-1}{240b_{i+1}^2}
  \left(b_{i+1}^2+1+20G_{i+1}\right).
\end{equation}
In particular,
\begin{equation}\label{eq:d-variance}
 \E[d_i^2]=\tau^2(b_{i+1}).
\end{equation}
\end{proposition}

\begin{proof}
Fix $i\geq0$. We again use the shorthands in \eqref{eq:v71-auto-0268}, together with $J=\1_{\{V<U\}}$.
The variables $X_i$, $\psi_i$, and $\psi_{i+1}$ are bounded and
$\cF_i$-measurable.  Hence so is $d_i$.  Using
\eqref{eq:conditional-X} and \eqref{eq:conditional-psi}, we obtain
\begin{equation}\label{eq:v71-auto-0024}
 \E[d_i\mid\cF_{i+1}]
 =\left(\frac{B-1}{2B}+\frac1{2B}-\frac12\right)\psi_{i+1}=0
\end{equation}
and \eqref{eq:reverse-MD} is proven.

For \eqref{eq:conditional-d2},
let us condition on
$\cF_{i+1}$.  This fixes $U,V$, and $J$, while $A,E$ remain independent
uniform digits in $\{0,\ldots,B-1\}$.  Equations
\eqref{eq:Xi-product-explicit}, \eqref{eq:X-psi-filtration},
\eqref{eq:one-mean}, \eqref{eq:psi-short}, and \eqref{eq:d-def} give
\begin{equation}\label{d_i_cases}
 d_i=
 \begin{cases}
 A-\mu(B)+\dfrac{J+U}{2}
   -\dfrac{(E+\tfrac12)(A+U)}{B},&A<E,\\[4pt]
 A-\mu(B)+\dfrac{U}{2}
   -\dfrac{(A+\tfrac12)(A+U)}{B},&A=E,\\[4pt]
 E-\mu(B)+\dfrac{1-J+U}{2}
   -\dfrac{(E+\tfrac12)(A+U)}{B},&A>E.
 \end{cases}
\end{equation}
To organise the computation of the conditional second moment of $d_i$, we introduce
\begin{align}
 g_{A,E}&:=\min\!\left(A,E+\frac12\right)
       -\frac{A(E+\tfrac12)}{B},\label{eq:v71-auto-0200}\\
 q_E&:=\frac12-\frac{E+\tfrac12}{B},\qquad
 s_{A,E}:=\1_{\{A<E\}}-\1_{\{E<A\}},\label{eq:v71-auto-0201}
\end{align}
so that \eqref{d_i_cases} becomes $d_i=(g_{A,E}-\mu(B))+Uq_E+\frac J2s_{A,E}$. We treat $d_i$ as a sum of \emph{three} terms (and hence $d_i^2$ will be written as a sum of six terms, see \eqref{eq:v71-auto-0202}-\eqref{eq:v71-auto-0203} below).
A direct computation gives
 
\begin{align}
 \sum_{A=0}^{B-1}g_{A,E}
 &=\frac{B-E-\tfrac12}{B}\sum_{A=0}^{E}A
   +\frac{E+\tfrac12}{B}\sum_{A=E+1}^{B-1}(B-A)=\frac{B-1}{4}+\frac{E(B-1-E)}2,\label{sum_over_A_of_g_AE}
\end{align}
implying 
\begin{align}
\label{sum_over_AE_of_g_AE}
 \sum_{A,E=0}^{B-1}g_{A,E}&=B^2\mu(B).
\end{align}
Similarly,
\begin{align}
 \sum_{A,E=0}^{B-1}g_{A,E}^2 &
 =\frac{(B^2-1)(4B^2-1)}{360}.\label{sum_g_AE^2}
\end{align}
It follows from \eqref{sum_over_AE_of_g_AE} and \eqref{sum_g_AE^2} that
\begin{equation}
 \frac{1}{B^2}\sum_{A,E=0}^{B-1}(g_{A,E}-\mu(B))^2
 =\frac{(B^2-1)(4B^2-1)}{360B^2}-\frac{2}{B^2}\mu(B)\sum_{A,E=0}^{B-1}g_{A,E}+\mu(B)^2
 =\frac{B^4-1}{240B^2}.\label{eq:avg(g_AE-mu(B))^2}
\end{equation}
Note that \eqref{sum_over_A_of_g_AE}  is invariant under
$E\mapsto B-1-E$, whereas $q_{B-1-E}=-q_E$.  Therefore
\begin{equation}\label{eq:avg(g_AE-mu(b))q_E}
 \frac{1}{B^2}\sum_{A,E=0}^{B-1}(g_{A,E}-\mu(B))q_E
 =
 0.
\end{equation}
Trivially, $\sum_{A,E=0}^{B-1}s_{A,E}=0$. Further computations give
\begin{align}
   & \frac{1}{B^2}\sum_{A,E=0}^{B-1}(g_{A,E}-\mu(B))s_{A,E}=  \frac{1}{B^2}\sum_{A,E=0}^{B-1}g_{A,E}s_{A,E}=-\frac{(B-1)(2B-1)}{12B^2},\label{eq:avg(g_AE-mu(B))s_AE}\\
    &\frac{1}{B^2}\sum_{A,E=0}^{B-1} q_{E}^2=\frac{B^2-1}{12B^2},\label{eq:avgq_E^2}\\
    &\frac{1}{B^2}\sum_{A,E=0}^{B-1} s_{A,E}^2=\frac{B-1}{B},\label{eq:avgs_AE^2}\\
    &\frac{1}{B^2}\sum_{A,E=0}^{B-1} q_E s_{A,E}=-\frac{2}{B^2}\sum_{A,E=0}^{B-1} q_E^2=-\frac{B^2-1}{6B^2}.\label{eq:avgq_E*s_AE}
\end{align}
Combining \eqref{eq:avg(g_AE-mu(B))^2}--\eqref{eq:avgq_E*s_AE} and using the fact that $J^2=J$, we have
\begin{align}
\E[d_i^2\mid\cF_{i+1}]
 &=\frac{1}{B^2}\sum_{A,E=0}^{B-1}\left((g_{A,E}-\mu(B))^2+2U (g_{A,E}-\mu(B))q_E\right.\label{eq:v71-auto-0202}\\
 &\left.\qquad\qquad\qquad +J(g_{A,E}-\mu(B))s_{A,E}+U^2 q_E^2+UJq_E s_{A,E}+\frac{J}{4}s_{A,E}^2\right)\label{eq:v71-auto-0203}\\
 &=\frac{B^4-1}{240B^2}+0
   -J\frac{(B-1)(2B-1)}{12B^2}+U^2\frac{B^2-1}{12B^2}-UJ\frac{B^2-1}{6B^2}+J\frac{B-1}{4B}\label{eq:v71-auto-0204}\\
 &=\frac{B^2-1}{240B^2}
   \left(B^2+1+20\{U^2+J(1-2U)\}\right).\label{eq:v71-auto-0205}
\end{align}
The expression in braces is $G_{i+1}$, proving
\eqref{eq:conditional-d2}.
To obtain \eqref{eq:d-variance}, it remains to average this identity over $U, V$ (independent uniform random variables on $[0,1)$). We get
\begin{equation}\label{eq:v71-auto-0025}
 \E[G_i]
 =\E[U_i(a)^2]+\int_0^1u(1-2u)\,du
 =\frac13-\frac16=\frac16.
\end{equation}
Consequently,
\begin{equation}\label{eq:v71-auto-0026}
 \E[d_i^2]
 =\frac{B^2-1}{240B^2}
   \left(B^2+1+\frac{20}{6}\right)
 =\frac{3B^4+10B^2-13}{720B^2}
 =\tau^2(B),
\end{equation}
proving \eqref{eq:d-variance}.
\end{proof}

We now compute the covariance structure of the $\Xi_i$'s from  \eqref{eq:Xi-product-explicit}. Recall  \eqref{eq:Pm-def}, \eqref{eq:M_and_T_and_V}, and \eqref{eq:SQDelta}.

\begin{proposition}\label{prop:exact-covariance}
For $0\leq i<j$,
\begin{equation}\label{eq:Xi-covariance}
 \Cov(\Xi_i,\Xi_j)=
 -\frac{(b_{i+1}-1)(b_{j+1}-1)}
 {24\, b_{i+1}\cdots b_{j+1}
}.
\end{equation}
Furthermore, for every $m\geq1$,
\begin{equation}\label{eq:product-exact-var}
 \Var\left(\mathcal P_m\right)=V_m
\end{equation}
and we have the bounds \begin{equation}\label{eq:variance-comparison}
 \frac1{240}S_m\leq V_m\leq\frac1{36}S_m. 
\end{equation}
\end{proposition}

\begin{proof}
Recalling \eqref{eq:X-psi-filtration}, we write
$\Cov(\Xi_i,\Xi_j)=\E[X_iX_j]$.  For $0\le i<j$, the variable $X_j$ is
$\cF_{i+1}$-measurable.  Hence \eqref{eq:conditional-X} gives
\begin{align}
 \E[X_iX_j]
 &=\frac{b_{i+1}-1}{2\,b_{i+1}}\E[\psi_{i+1}X_j].
 \label{eq:covariance-first-conditioning}
\end{align}
Iterating \eqref{eq:conditional-psi} from level $i+1$ through level
$j-1$ gives
\begin{equation}\label{eq:v71-auto-0027}
 \E[\psi_{i+1}\mid\cF_j]
 =\frac{\psi_j}{b_{i+2}\cdots b_j}.
\end{equation}
Substituting \eqref{eq:v71-auto-0027} into
\eqref{eq:covariance-first-conditioning}, and then using
\eqref{eq:psi-X} at level $j$, yields
\begin{align}
 \E[X_iX_j]
 &=\frac{b_{i+1}-1}
 {2\,b_{i+1}b_{i+2}\cdots b_j}\E[\psi_jX_j]=-\frac{(b_{i+1}-1)(b_{j+1}-1)}
 {24\,b_{i+1}\cdots b_{j+1}
}
,\label{eq:v71-auto-0207}
\end{align}
which proves \eqref{eq:Xi-covariance}.

We now compute the variance of the sum of the $\Xi_i$'s by using \eqref{eq:martingale-coboundary}.  First, distinct
martingale differences are orthogonal: if $i<j$, then $d_j$ is
$\cF_{i+1}$-measurable and
\begin{equation}\label{eq:v71-auto-0029}
 \E[d_id_j]=\E\left[d_j\E[d_i\mid\cF_{i+1}]\right]=0.
\end{equation}
Next, \eqref{eq:conditional-psi} and \eqref{eq:psi-second-moment}  imply
\begin{equation}\label{eq:v71-auto-0030}
 \E[\psi_i\psi_{i+1}]
 =\E\left[\psi_{i+1}\E[\psi_i\mid\cF_{i+1}]\right]
 =\frac1{6\,b_{i+1}}.
\end{equation}
Consequently, by \eqref{eq:d-def}, \eqref{eq:psi-X}, and \eqref{eq:v71-auto-0030}, we see that $d_i$ and $\psi_i$ are also orthogonal:
\begin{align}
 \E[d_i\psi_i]
 &=\E[X_i\psi_i]
   +\frac12\left(\E[\psi_i^2]-\E[\psi_i\psi_{i+1}]\right)=-\frac{b_{i+1}-1}{12\,b_{i+1}}
   +\frac{1}{2}\left(\frac{1}{6}-\frac{1}{6\,b_{i+1}}\right)=0.\label{eq:v71-auto-0209}
\end{align}
For $i<m$, the variable $\psi_m$ is $\cF_{i+1}$-measurable, and hence by \eqref{eq:reverse-MD} we have
\begin{align}
\E[d_i\psi_m]=\E[\E[d_i\psi_m|\cF_{i+1}]]=\E[\psi_m\E[d_i|\cF_{i+1}]]=0.\label{E[d_i_psi_m]=0}
\end{align}
Iterating \eqref{eq:conditional-psi} in the other
direction gives
\begin{equation}\label{eq:psi-zero-conditioning}
 \E[\psi_0\mid\cF_i]=B_i^{-1}\psi_i.
\end{equation}
Since $d_i$ is $\cF_i$-measurable, \eqref{eq:psi-zero-conditioning} and \eqref{eq:v71-auto-0209} yield
\begin{equation}\label{eq:v71-auto-0031}
 \E[d_i\psi_0]=\E[d_i\E[\psi_0|\cF_{i}]]
 =B_i^{-1}\E[d_i\psi_i]=0.
\end{equation}
Thus, by \eqref{E[d_i_psi_m]=0} and \eqref{eq:v71-auto-0031}, the reverse-martingale-difference sum $\sum_{i=0}^{m-1} d_i$ in \eqref{eq:martingale-coboundary} is orthogonal
to both boundary variables $\psi_m$ and $\psi_0$.

It remains to calculate the boundary variance.  Taking $i=m$ in
\eqref{eq:psi-zero-conditioning} gives
\begin{equation}\label{eq:v71-auto-0032}
 \E[\psi_0\psi_m]=\E[\psi_m\E[\psi_0|\cF_{m}]]=\frac1{B_m}\E[\psi_m^2]=\frac1{6\,B_m},
\end{equation}
while $\E[\psi_0]=\E[\psi_m]=0$ and
$\E[\psi_0^2]=\E[\psi_m^2]=1/6$.  Therefore
\begin{equation}\label{eq:v71-auto-0033}
 \Var\left(\frac{\psi_m-\psi_0}{2}\right)
 =\frac{1-B_m^{-1}}{12}.
\end{equation}
Combining \eqref{eq:v71-auto-0033} with \eqref{eq:v71-auto-0029} and \eqref{eq:d-variance}, we obtain
\begin{align}
 \Var\left(\mathcal P_m\right)
 &=\Var\left(\sum_{i=0}^{m-1}X_i\right)=\sum_{i=0}^{m-1}\E[d_i^2]+\frac{1-B_m^{-1}}{12}=T_m+\frac{1-B_m^{-1}}{12}=V_m\label{eq:v71-auto-0212}
\end{align}
and  \eqref{eq:product-exact-var} is proven.

Finally, using the definition \eqref{eq:mu_and_tau} and the fact that  $b_{i+1}\geq2$ for every $i\geq0$,  we have the elementary bounds
\begin{equation}\label{eq:v71-auto-0034}
  \frac{b_{i+1}^2}{240}\leq\tau^2(b_{i+1})\leq\frac{b_{i+1}^2}{144}.
\end{equation}
Summing over $0\le i\le m-1$ gives
\begin{equation}\label{eq:v71-auto-0035}
 \frac{S_m}{240}\leq T_m\leq\frac{S_m}{144}.
\end{equation}
Since $S_m\geq4$ for every $m\geq1$ and
$\frac{(1-B_m^{-1})}{12}\leq\frac{1}{12}\leq \frac{S_m}{48}$, we conclude that $\frac{S_m}{240}\leq V_m
 \leq\frac{S_m}{144}+\frac{S_m}{48}
 =\frac{S_m}{36}$, proving \eqref{eq:variance-comparison}.
\end{proof}

\Cref{prop:exact-covariance} identifies the exact variance of the Haar-model
sum. For the martingale CLT, we must also control the \emph{predictable
quadratic variation} of the sequence $(d_i)_{i\geq0}$, namely\footnote{The
index $i+1$ on the $\sigma$-algebra is due to the reverse filtration.}
\begin{align}\label{predictable-quadratic-variation-def}
    \sum_{i=0}^{m-1} \E[d_i^2\mid\cF_{i+1}].
\end{align}
We compare \eqref{predictable-quadratic-variation-def} with the deterministic
sum
\begin{align}
    \sum_{i=0}^{m-1} \E[d_i^2].\label{var-martingale-sum}
\end{align}
By \eqref{eq:d-variance} and \eqref{eq:M_and_T_and_V}, this sum is the expectation of \eqref{predictable-quadratic-variation-def} and  equals
$T_m$. By the orthogonality relation \eqref{eq:v71-auto-0029}, the sum  also equals
$\Var\!\left(\sum_{i=0}^{m-1}d_i\right)$. By
\eqref{eq:conditional-d2}, the only randomness in
$\E[d_i^2\mid\cF_{i+1}]$ is contained in $G_{i+1}$. The next lemma controls how the random variable \eqref{predictable-quadratic-variation-def} deviates from its mean \eqref{var-martingale-sum} by quantifying its variance. To this end,  it suffices to quantify the covariance of $G_i$
and $G_j$ for $0\leq i<j$.

\begin{lemma}\label{lem:G-decorrelation}
We have
\begin{equation}\label{eq:G-cov-bound}
 |\Cov(G_i,G_j)|=O(2^{-(j-i)/2})
\end{equation}
uniformly for all $0\leq i<j$.
Furthermore, one has
\begin{equation}\label{eq:conditional-variance-fluctuation}
 \Var\left(\sum_{i=0}^{m-1}
 \E[d_i^2\mid\cF_{i+1}]\right)=
 O(m)
\end{equation}
uniformly in $m\geq1$.
\end{lemma}

\begin{proof}
Fix $0\le i<j$ and put $R=b_{i+1}\cdots b_j$. In particular, $R\geq 2^{j-i}$. Iterating
\eqref{Ui_and_Ui+1}, we may write
\begin{equation}\label{eq:tail-prefix-split}
 U_i(a)=\frac{P_a+U_j(a)}{R},
 \qquad
 U_i(e)=\frac{P_e+U_j(e)}{R},
\end{equation}
where $P_a$ (respectively $P_e$) belongs to $\{0,\ldots,R-1\}$ and depends only on the digits of $a$ (respectively of $e$) with
indices $i,\ldots,j-1$.  These two prefix integers are independent and
uniformly distributed on $\{0,\ldots,R-1\}$.  
We consider the rational truncations to the tails \eqref{eq:tail-prefix-split}: $\frac{P_a}{R}$ and $\frac{P_e}{R}$ and use them to define a truncated version of \eqref{eq:G-def}. That is, we set
\begin{align}
    G_i^{(j)}:=\left(\frac{P_a}{R}\right)^2+\1_{\{P_e<P_a\}}\left(1-2\frac{P_a}{R}\right).\label{def:G_i^j}
\end{align}
The random quantity $G_i^{(j)}$ is determined by the digits of $a$ and $e$ with
indices $i,\ldots,j-1$, so it is independent of $G_j$. We wish to compare $G_i$ and $G_i^{(j)}$. To this end, first compare $\1_{\{U_i(e)<U_i(a)\}}$  and $\1_{\{P_e<P_a\}}$. If $P_a\neq P_e$, then $U_i(e)<U_i(a)$ if and only if $P_e<P_a$. This is because $0\le U_j(a), U_j(e)<1$. Hence, the indicators $\1_{\{U_i(e)<U_i(a)\}}$ can only differ on the event $\{P_a=P_e\}$, whose probability is $\frac{1}{R}$.
Secondly, note that for $\epsilon\in\{0,1\}$, the functions $u\mapsto u^2+\epsilon (1-2u)$ are both $2$-Lipschitz on $[0,1]$ and take value in $[0,1]$. Since $|U_i(a)-\frac{P_a}{R}|< R^{-1}$, it follows $|G_i-G_i^{(j)}|\leq 2R^{-1}$ on $\{P_a\neq P_e\}$, whereas  $|G_i-G_i^{(j)}|\leq 1$ on  $\{P_a=P_e\}$.
Therefore,  
\begin{align}\label{L2-distance-squared-Gi-Gi^j}
    \E\left[\left(G_i-G_i^{(j)}\right)^2\right]
    \le 4 R^{-2}+R^{-1}\le 3R^{-1}\le 3\cdot 2^{-(j-i)}
\end{align}
since $R\ge2$ and $R\ge2^{j-i}$. We use the independence of $G_i^{(j)}$ and $G_j$, Cauchy--Schwarz, and the bound $\Var(G_j)=\E[(G_j-\E[G_j])^2]\le\frac{1}{4}$ to obtain
\begin{align}
 |\Cov(G_i,G_j)|
 =|\Cov(G_i-G_i^{(j)},G_j)|
 \leq \frac{1}{2}\left(\E\left[\left(G_i-G_i^{(j)}\right)^2\right]\right)^{1/2}.\label{eq:v71-auto-0213}
\end{align}
Combining the latter bound and 
 \eqref{L2-distance-squared-Gi-Gi^j} we get
 \eqref{eq:G-cov-bound} with an implied constant $C=\sqrt{3}/2$.
Using \eqref{eq:conditional-d2} and \eqref{eq:d-variance} we write
\begin{equation}\label{eq:v71-auto-0269}
 H_i:=\E[d_i^2\mid\cF_{i+1}]-\E[d_i^2]
 =\frac{b_{i+1}^2-1}{12b_{i+1}^2}
 \left(G_{i+1}-\frac{1}{6}\right).
\end{equation}
Since
\begin{equation}\label{eq:v71-auto-0037}
 0\leq\frac{b_{i+1}^2-1}{12b_{i+1}^2}\leq\frac{1}{12},
\end{equation}
we have
\begin{equation}\label{eq:v71-auto-0038}
 \Var(H_i)\leq\frac{1}{12^2}\frac{1}{4}=\frac{1}{576}.
\end{equation}
Similarly, for $i<j$, \eqref{eq:G-cov-bound} gives
\begin{equation}\label{eq:v71-auto-0039}
 |\Cov(H_i,H_j)|
 \leq\frac{1}{12^2}\frac{\sqrt{3}}{2}2^{-(j-i)/2}.
\end{equation}
We also use the summability estimate
\begin{equation}\label{eq:v71-auto-0040}
 \sum_{0\leq i<j\leq m-1}2^{-(j-i)/2}
 \leq(1+\sqrt{2})m.
\end{equation}
Therefore
\begin{align}
&\Var\left(\sum_{i=0}^{m-1}
 \E[d_i^2\mid\cF_{i+1}]\right)=\Var\left(\sum_{i=0}^{m-1}H_i\right)\\
 &=\sum_{i=0}^{m-1} \Var(H_i)+2\sum_{0\le i<j\le m-1}\Cov(H_i,H_j)\le\left(\frac{1}{576}+\frac{\sqrt{3}}{144}(1+\sqrt{2})\right)m,\label{eq:v71-auto-0214}
\end{align}
which proves \eqref{eq:conditional-variance-fluctuation}. 
\end{proof}

Now recall the definition \eqref{eq:SQDelta} of  $S_m$, $Q_m$, and $\Delta_m$. The next lemma shows that the natural assumption 
\begin{equation}\label{eq:feller}
 \lim_{m\to\infty}\Delta_m=0,
\end{equation}
is equivalent to a standard Feller condition. 

\begin{lemma}\label{lem:feller-equivalence}
For every $m\geq1$, set $M_m^*:=\displaystyle\max_{1\leq j\leq m}b_j^2$.  Then
\begin{equation}\label{eq:feller-sandwich}
 \left(\frac{M_m^*}{S_m}\right)^2
 \leq\Delta_m\leq\frac{M_m^*}{S_m}.
\end{equation}
Hence \eqref{eq:feller} holds if and only if 
$\displaystyle\lim_{m\to\infty}\frac{M_m^*}{S_m}=0$.
\end{lemma}

\begin{proof}
Choose $j_*$ such that $b_{j_*}^2=M_m^*$.  Since $Q_m$ is a sum of
nonnegative terms, $Q_m\geq b_{j_*}^4=(M_m^*)^2$.
On the other hand, $b_j^4\leq M_m^*b_j^2$ for every $j\leq m$, and hence $Q_m=\sum_{j=1}^m b_j^4
 \leq M_m^*S_m$.
Dividing these inequalities by $S_m^2$ proves
\eqref{eq:feller-sandwich} and the asserted
equivalence is immediate.
\end{proof}

\section{A quantitative CLT in the Haar digit model}
In this section we prove a central limit theorem for the sequence of random variables $(\mathcal P_m)_{m\geq 1}$ defined in \eqref{eq:Pm-def}.
Recall the deterministic sequences $M_m$, $V_m$, $\Delta_m$ from \eqref{eq:M_and_T_and_V} and \eqref{eq:SQDelta}.

\begin{theorem}[CLT in the Haar digit model]\label{thm:product-clt}
We have
\begin{equation}\label{eq:product-rate}
 \dK\!\left(\frac{\mathcal P_m-M_m}{\sqrt{V_m}}\:,\:\Normal(0,1)\right)
 =O(\Delta_m^{1/5})
\end{equation}
uniformly in $m\geq1$. In particular, \eqref{eq:feller} implies the central limit theorem. Assuming 
\eqref{eq:feller}, for every fixed $p>0$, we have
\begin{equation}\label{eq:product-moments}
 \E\left[\left|\frac{\mathcal P_m-M_m}{\sqrt{V_m}}\right|^p\right]
 \longrightarrow \E[|G|^p],
 \qquad G\sim\Normal(0,1),
\end{equation}
as $m\to\infty$.
\end{theorem}

\begin{proof}
For a fixed $m$, reverse the order of the differences by setting
\begin{equation}\label{eq:v71-auto-0043}
 d_{m,k}=d_{m-k},
 \qquad
 \cG_{m,k}=\cF_{m-k},
 \qquad 1\leq k\leq m,
\end{equation}
and set $\cG_{m,0}=\cF_m$.  The filtration $(\cG_{m,k})_{0\leq k\leq m}$
is increasing, and \eqref{eq:reverse-MD} gives
\begin{equation}\label{eq:v71-auto-0044}
 \E[d_{m,k}\mid\cG_{m,k-1}]
 =\E[d_{m-k}\mid\cF_{m-k+1}]=0.
\end{equation}
Thus $(d_{m,k},\cG_{m,k})_{1\le k\le m}$ is a (forward) martingale-difference array.
Moreover,
\begin{equation}\label{eq:v71-auto-0045}
 s_m^2:=\sum_{k=1}^m\E[d_{m,k}^2]
 =\sum_{i=0}^{m-1}\E[d_i^2]=T_m
\end{equation}
by \eqref{eq:d-variance}.  We may therefore apply the $p=2$ case of the
martingale Berry--Esseen estimate of Haeusler~\cite{Haeusler}.  See also
Mourrat~\cite{Mourrat} for related quantitative formulations.  Using the original index $i$ rather than $k$, it reads
\begin{equation}\label{eq:Haeusler}
 \dK\!\left(\frac{1}{s_m}\displaystyle\sum_{ i=0}^{m-1}d_i\:,\:\Normal(0,1)\right)
 \leq C\left(
 \left\|\frac1{s_m^2}\sum_{i<m}\E[d_i^2\mid\cF_{i+1}]-1\right\|_2^2
 +\frac1{s_m^4}\sum_{i<m}\E[d_i^4]
 \right)^{\!\!1/5}
\end{equation}
for some absolute constant $C>0$. 
We estimate the two terms on the right-hand side of \eqref{eq:Haeusler} separately.  Since $\E[\sum_{i=0}^{m-1}\E[d_i^2|\cF_{i+1}]]=\sum_{i=0}^{m-1}\E[d_i^2]=T_m$ by \eqref{eq:d-variance} and \eqref{eq:M_and_T_and_V}, we can write
\begin{align}
 \left\|\frac1{T_m}\sum_{i<m}
  \E[d_i^2\mid\cF_{i+1}]-1\right\|_2^2&=\frac{1}{T_m^2}\E\left[\left(\sum_{i=0}^{m-1}\E[d_i^2|\cF_{i+1}]-T_m\right)^{\!\!2\:}\right]\label{eq:v71-auto-0214_2}\\
 &=\frac1{T_m^2}\Var\!\left(\sum_{i=0}^{m-1}
   \E[d_i^2\mid\cF_{i+1}]\right).\label{eq:v71-auto-0215}
\end{align}
Now the statement \eqref{eq:conditional-variance-fluctuation} of
\Cref{lem:G-decorrelation}, the lower bound $T_m\geq S_m/240$ from
\eqref{eq:v71-auto-0035}, and the trivial lower bound $Q_m\ge 16m$ imply
\begin{align}
 \frac1{T_m^2}\Var\!\left(\sum_{i=0}^{m-1}
   \E[d_i^2\mid\cF_{i+1}]\right)
 &=O\!\left(\frac{m}{S_m^2}\right)
 =O\!\left(\frac{Q_m}{S_m^2}\right)=O(\Delta_m)\label{eq:v71-auto-0216}
\end{align}
uniformly for all $m\geq1$.

We now estimate the fourth-moment term in \eqref{eq:Haeusler}. Note that
$|\Xi_i|\leq b_{i+1}$ by \eqref{eq:xi-bound}, and
$|X_i|\leq2b_{i+1}$ by \eqref{eq:one-mean}.  Also, $|\psi_i|\leq1$.  Hence, by
\eqref{eq:d-def} and $b_{i+1}\geq2$, we have
\begin{equation}\label{eq:d-pointwise-bound}
 |d_i|\leq2b_{i+1}+1\leq3b_{i+1}.
\end{equation}
This fact, along with \eqref{eq:v71-auto-0035}, implies 
\begin{equation}\label{eq:v71-auto-0046}
 \frac1{T_m^2}\sum_{i=0}^{m-1}\E[d_i^4]
 =O\!\left(\frac{\sum_{j=1}^m b_j^4}{S_m^2}\right)
 =O(\Delta_m)
\end{equation}
uniformly for all $m\geq1$.
Substitution into \eqref{eq:Haeusler} proves
\begin{equation}\label{eq:martingale-rate-intermediate}
 \dK\left(\frac{\sum_{i<m}d_i}{\sqrt{T_m}}\:,\:\Normal(0,1)\right)
 =O(\Delta_m^{1/5}).
\end{equation}

We now restore the coboundary and the variance normalization.
Using \eqref{eq:martingale-coboundary}, we rewrite
\begin{equation}\label{eq:product-normalized-transfer}
 \frac{\mathcal P_m-M_m}{\sqrt{V_m}}
 =\lambda_m\frac{\sum_{i=0}^{m-1}d_i}{\sqrt{T_m}}+\frac{R_m}{\sqrt{V_m}},
\end{equation}
where
\begin{equation}\label{eq:v71-auto-0047}
 R_m=\frac12(\psi_m-\psi_0),
 \qquad
 \lambda_m=\sqrt{\frac{T_m}{V_m}}.
\end{equation}
Note that $|R_m|\leq1$. Moreover,
%
since $0\leq V_m-T_m\leq1/12$ and $T_m\ge S_m/240$, we have
\begin{equation}\label{eq:v71-auto-0048}
 |\lambda_m-1|\leq\frac{10}{S_m},
 \qquad
 \frac{|R_m|}{\sqrt{V_m}}\leq\sqrt{\frac{240}{S_m}}
\end{equation}
for all $m\geq1$.
Now we control the distribution function of $\frac{\mathcal P_m-M_m}{\sqrt{V_m}}$ above and below by that of $\lambda_m\frac{\sum_{i=0}^{m-1}d_i}{\sqrt{T_m}}$. Specifically, we use the bounds
\begin{equation}\label{eq:v71-auto-0049}
 \Pp(W\leq t-\eta)\leq\Pp(W+R\leq t)
 \leq\Pp(W\leq t+\eta)
\end{equation}
with random variables $W=\lambda_m\sum_{i=0}^{m-1}d_i/\sqrt{T_m}$,
$R=R_m/\sqrt{V_m}$, and
$|R|\le \eta:=\sqrt{240/S_m}$.
Note that 
\begin{align}
\left|\Pp(W\le  \tau)-\Phi(\tau)\right|\le \left|\Pp\!\left(\frac{\sum_{i=0}^{m-1}d_i}{\sqrt{T_m}}\le  \frac{\tau}{\lambda_m}\right)-\Phi\!\left(\frac{\tau}{\lambda_m}\right)\right|+\left|\Phi\!\left(\frac{\tau}{\lambda_m}\right)-\Phi(\tau)\right|.\label{compare-with-Phi-add-and-subtract}
\end{align}
Taking the supremum over $\tau\in\R$ in \eqref{compare-with-Phi-add-and-subtract}, using \eqref{eq:martingale-rate-intermediate}, and noting that $\lambda_m$ is bounded away from zero, we obtain $\delta_m:=\dK(W,\Normal(0,1))=O (\Delta_m^{1/5}+|\lambda_m-1|)$ uniformly for all $m\ge 1$.
Hence the first inequality of \eqref{eq:v71-auto-0048} and $S_m^{-1}\le \Delta_m^{1/5}$ yield $\delta_m=O(\Delta_m^{1/5})$. Moreover, $Q_m\geq4S_m$, so $\Delta_m\geq4/S_m$; since $S_m\geq4$, this gives $\eta=O(\Delta_m^{1/5})$.
The rightmost quantity in \eqref{eq:v71-auto-0049} satisfies the bound
\begin{align}
    \Pp(W\leq t+\eta)\le \Phi(t+\eta)+\delta_m\le \Phi(t)+\delta_m+\tfrac{1}{\sqrt{2\pi}}\eta=\Phi(t)+O(\Delta_m^{1/5})\label{pdf-upper-bound}
\end{align}
after using the second inequality of \eqref{eq:v71-auto-0048}.
Similarly, the  leftmost quantity in \eqref{eq:v71-auto-0049} satisfies
\begin{align}
    \Pp(W\leq t-\eta)\ge \Phi(t-\eta)-\delta_m\ge \Phi(t)-O(\Delta_m^{1/5}). \label{pdf-lower-bound}
\end{align}
Combining \eqref{eq:v71-auto-0049}, \eqref{pdf-upper-bound}, and \eqref{pdf-lower-bound}, and taking the supremum over $t\in\mathbb{R}$ proves the first part of the theorem, \eqref{eq:product-rate}. In particular, if $\Delta_m\to0$ as $m\to\infty$, we get the central limit theorem $\frac{\mathcal P_m-M_m}{\sqrt{V_m}}\Rightarrow\Normal(0,1)$.

It remains to prove convergence of moments. This will follow from uniform
integrability, along with the already established convergence in
distribution. We fix $m\ge 1$ and apply Burkholder's square-function
inequality~\cite[Theorem~2.10]{HallHeyde} to the (forward)
martingale-difference array $(d_{m,k})_{1\leq k\leq m}$: for every $q\geq2$
there is a constant $C_q>0$, depending only on $q$, for which
\begin{equation}\label{eq:uniform-moments}
 \E\!\left[\left|\sum_{k=1}^{m}d_{m,k}\right|^{q\:}\right]\leq C_q\, \E\!\left[\left(\sum_{k=1}^m d_{m,k}^2\right)^{q/2}\right].
\end{equation}
The pointwise
bound \eqref{eq:d-pointwise-bound} gives
$\sum_{i=0}^{m-1}d_i^2\leq9S_m$ and  implies 
\begin{equation}\label{eq:uniform-moments-2}
 \E\!\left[\left|\sum_{k=1}^{m}d_{m,k}\right|^{q\:}\right]
 \leq C_q(9S_m)^{q/2}.
\end{equation}  Using again $|R_m|\leq1$ and the lower bound 
$V_m\geq S_m/240$ from \eqref{eq:variance-comparison}, we see that  \eqref{eq:v71-auto-0043}, \eqref{eq:martingale-coboundary}, and
\eqref{eq:uniform-moments-2} imply
\begin{equation}\label{eq:v71-auto-0051}
 \sup_{m\ge 1}\E\left[\left|\frac{\mathcal P_m-M_m}{\sqrt{V_m}}\right|^q\right]<\infty
\end{equation}
for all $q\ge2$.
Now, given a fixed $p>0$, choose $q>\max\{p,2\}$.  The  $L^q$ bound \eqref{eq:v71-auto-0051} makes the $p$th powers uniformly integrable.  Combining this with the CLT
proves \eqref{eq:product-moments}.
\end{proof}

\section{Clock coupling for uniform random times}

The transfer from the Haar digit model to the original statistic $\cD_H$
is implemented by a maximal coupling of the clock digits associated with $N_H$. With $m=m(H)$
and $1\leq r\leq m$, \Cref{lem:time-digits-Haar} couples
$N_H\bmod B_r$ to a uniform residue, and hence the first $r$ digits of
$c(N_H)$ to those of a Haar clock string, with failure probability
$\eps_{H,r}$ in the sense of \Cref{rem:maximal-coupling}. Using the same independent Haar endpoint in both models,
\Cref{lem:clock-coupling} below controls $\cD_H-\mathcal P_m$ in $L^p$ at the
scale $\sqrt{S_m}\,\rho_{H,r,m}^{(p)}$. For $p=2$ this transfers the exact
Haar-model mean \eqref{eq:product-exact-mean} and variance
\eqref{eq:product-exact-var} to estimates for the mean and standard deviation
of $\cD_H$. 
In \Cref{prop:general-time}, the \(L^2\) comparison furnished by
\Cref{lem:clock-coupling} is used to show that the normalized random-time
statistic and the normalized Haar-model statistic have close distribution
functions.  Combining this comparison with the quantitative central limit
theorem for the Haar model in \Cref{thm:product-clt} eventually proves
\Cref{thm:main}.  The analogous \(L^q\) comparisons provide the uniform
integrability needed for the moment convergence in \Cref{thm:moments}.

\begin{lemma}[Coupling finite-support and Haar clocks]\label{lem:clock-coupling}
Let $p\geq1$, $H\geq2$, $m=m(H)$, and $1\leq r\leq m$.  On a common
probability space one can construct $\cD_H$ and a copy of $\mathcal P_m$
such that
\begin{equation}\label{eq:coupling-Lp}
 \|\cD_H-\mathcal P_m\|_p
 \leq C_p\!\left(
 1+\sum_{j=r}^m b_j+
 \left(\sum_{j=1}^m b_j\right)\eps_{H,r}^{1/p}
 \right)
\end{equation}
uniformly in $H$ and $r$, where $C_p$ depends only on $p$.
\end{lemma}

\begin{proof}
Put $q=B_r$.  The first $r$ mixed-radix digits of $N_H$ are determined by
$R:=N_H\bmod q$.  As explained in \Cref{rem:maximal-coupling}, choose a
uniform random residue $R^*$ in $\{0,\ldots,q-1\}$ on the same probability
space so that $(R,R^*)$ is a maximal coupling.  Explicitly, if
$p_j:=\Pp(R=j)$, this coupling assigns mass $\min\{p_j,q^{-1}\}$ to the
diagonal pair $(j,j)$. Any off-diagonal residual matching described in
\Cref{rem:maximal-coupling} may be used, and the choice is immaterial below.
By
\Cref{lem:time-digits-Haar}, the event
\begin{equation}\label{eq:v71-auto-0052}
 \mathcal E=\{R=R^*\}
\end{equation}
has failure probability $\Pp(\mathcal E^c)=\eps_{H,r}$.  The digits of $R^*$ are
independent uniform digits in positions $0,\ldots,r-1$. Append independent
uniform digits in all positions $i\geq r$ to obtain a Haar clock string
$e^*=(e_i^*)_{i\geq0}$.
Independently of this coupling, choose one Haar endpoint string
$a=(a_i)_{i\ge 0}$, let $Y=\sum_{i\geq0}\frac{a_i}{B_{i+1}}$ be the corresponding endpoint.
In the coupled
realization below, the two statistics use this same endpoint string $a$;
only their clock strings differ.
%
Let $e^H:=c(N_H)$  and 
%
recall \eqref{eq:xi-original}. 
Define 
\begin{align}
 \xi_i^H&:=\xi_i(e^H,Y).\label{eq:v71-auto-0217}
 \end{align}
 Moreover, let
 \begin{align}
 \Xi_i^*&:=\min(a_i,e_i^*)
 +\1_{\{a_i<e_i^*\}}\1_{\{U_{i+1}(e^*)<U_{i+1}(a)\}}
 -\frac{e_i^*(a_i+U_{i+1}(a))}{b_{i+1}}.\label{eq:v71-auto-0218}
\end{align}
Thus $\Xi_i^*$ is the explicit product-model contribution
\eqref{eq:Xi-product-explicit} on the coupled space.  Set
\begin{equation}\label{eq:coupled-product-sum}
    \mathcal P_m:=\sum_{i=0}^{m-1}\Xi_i^*.
\end{equation}  Recall \eqref{H-sandwiched}. Since $N_H<B_m$, the digits of
$e^H$ vanish from level $m$ onward.  The digital decomposition \Cref{lem:digital-decomposition} therefore
gives
\begin{equation}\label{eq:coupled-local-sums}
 \cD_H=\sum_{i=0}^{m-1}\xi_i^{H}.
\end{equation}
When $N_H=0$, \eqref{eq:coupled-local-sums} still holds because both sides vanish.
Using the local normal form
\eqref{eq:local-normal-form} from  \Cref{lem:local-normal-form}, we rewrite \eqref{eq:v71-auto-0217} as

\begin{equation}\label{xi_i^H-rewritten}
\xi_i^H=\min(a_i,e_i^H)
 +\1_{\{a_i<e_i^H\}}\1_{\{U_{i+1}(e^H)<U_{i+1}(a)\}}
 -\frac{e_i^H(a_i+U_{i+1}(a))}{b_{i+1}}.
\end{equation}
On $\mathcal E$, the two clock strings $e^H$ and $e^*$ have the same digits in positions
$0,\ldots,r-1$.  Fix $0\leq i\leq r-2$. In this case, \eqref{eq:v71-auto-0218}  and \eqref{xi_i^H-rewritten} can only differ in the second term, in which the two indicators compare
$U_{i+1}(e^H)$ and $U_{i+1}(e^*)$ with the common endpoint tail
$U_{i+1}(a)$. 
Because the time tails share the digits in positions
$i+1,\ldots,r-1$, iterating \eqref{Ui_and_Ui+1} shows that their distance $|U_{i+1}(e^H)-U_{i+1}(e^*)|$ is
at most
\begin{equation}\label{eq:coupling-tail-distance}
 \delta_i:=\frac1{b_{i+2}\cdots b_r}.
\end{equation}
This implies that the random variables $\xi_i^H$ and $\Xi_i^*$ are different with probability at most $\delta_i$. To see this, note that 
conditional on the two clock strings, $U_{i+1}(a)$ remains uniform on $[0,1)$
and independent of them and,  for fixed $v,v'\in[0,1)$,
\begin{equation}\label{eq:v71-auto-0053}
 \Pp\left(\1_{\{v<U_{i+1}(a)\}}\neq
                 \1_{\{v'<U_{i+1}(a)\}}\right)=|v-v'|.
\end{equation}
It follows from \eqref{eq:coupling-tail-distance} that
\begin{equation}\label{eq:coupling-low-level-Lp}
 \left\|\1_{\mathcal E}(\xi_i^{H}-\Xi_i^*)\right\|_p
 \leq\delta_i^{1/p},
\qquad 0\le i\le r-2.
\end{equation}
Since every radix in \eqref{eq:coupling-tail-distance} is at least two,
\begin{equation}\label{eq:v71-auto-0054}
 \sum_{i=0}^{r-2}(b_{i+2}\cdots b_r)^{-1/p}
 \leq\sum_{k=1}^{\infty}2^{-k/p}=:C_p.
\end{equation}
This also covers the boundary case $r=1$, when the sum over low levels is empty.
For the remaining levels $r-1\le i\le m-1$, the bound \eqref{eq:xi-bound} from  \Cref{lem:local-normal-form} gives 
\begin{equation}\label{eq:v71-auto-0055}
 \sum_{i=r-1}^{m-1}|\xi_i^{H}-\Xi_i^*|
 \leq2\sum_{i=r-1}^{m-1}b_{i+1}
 =2\sum_{j=r}^m b_j
\end{equation}
pointwise.

On the failure event $\mathcal E^c$, we instead use the bound  \eqref{eq:xi-bound} for all levels:
\begin{equation}\label{eq:v71-auto-0056}
 \1_{\mathcal E^c}|\cD_H-\mathcal P_m|
 \leq2\cdot \1_{\mathcal E^c}\sum_{j=1}^m b_j.
\end{equation}
Apply Minkowski's inequality to the low-level part on $\mathcal E$, the
remaining levels on $\mathcal E$, and the whole difference on
$\mathcal E^c$.  Since
$\|\1_{\mathcal E^c}\|_p=\eps_{H,r}^{1/p}$, we obtain
\begin{equation}\label{eq:v71-auto-0057}
 \|\cD_H-\mathcal P_m\|_p
 \leq C_p+2\sum_{j=r}^m b_j
  +2\left(\sum_{j=1}^m b_j\right)\eps_{H,r}^{1/p},
\end{equation}
which proves \eqref{eq:coupling-Lp}.
\end{proof}

The next proposition provides the key  transfer principle: it converts the CLT proved in \Cref{thm:product-clt} for the ideal Haar-clock model $\mathcal P_m$ into a CLT for $D_N(Y)$ under a general random-time law for $N$.

\begin{proposition}[General random-time laws]\label{prop:general-time}
Let $m\geq1$, and let $N$ be any random variable supported on
$\{0,\ldots,B_m-1\}$ and independent of $Y$.  For $1\leq r\leq m$, put
\begin{equation}\label{eq:general-delta}
 \delta_{N,r}:=\TV\bigl(\mathcal L(N\bmod B_r),
       \operatorname{Unif}\{0,\ldots,B_r-1\}\bigr)
\end{equation}
and, for $p\geq1$, put
\begin{equation}\label{eq:general-rho}
 \rho_{N,r}^{(p)}:=
 \frac{1+\sum_{j=r}^m b_j+
       (\sum_{j=1}^m b_j)\delta_{N,r}^{1/p}}
      {\sqrt{S_m}}.
\end{equation}
Then, with $\cD_N=D_N(Y)$, one has
\begin{align}
 |\E[\cD_N]-M_m|&=O\!\left(\sqrt{S_m}\,\rho_{N,r}^{(2)}\right),\label{eq:general-mean}\\
 \bigl|\sqrt{\Var(\cD_N)}-\sqrt{V_m}\bigr|
 &=O\!\left(\sqrt{S_m}\,\rho_{N,r}^{(2)}\right),\label{eq:general-sd}\\
 \dK\left(\frac{\cD_N-M_m}{\sqrt{V_m}}\:,\:\Normal(0,1)\right)
 &=O\!\left(\Delta_m^{1/5}+
       (\rho_{N,r}^{(2)})^{2/3}\right).\label{eq:general-rate}
\end{align}
The implied constants are absolute and the estimates are uniform in the
parameters above.
Along any sequence $m\to\infty$, let $N=N_m$ and choose
$1\leq r=r(m)\leq m$.  Suppose that $\Delta_m\to0$ and that, for every fixed $p>0$, there exists $q>\max\{p,2\}$ such that 
$\rho_{N_m,r(m)}^{(q)}\to0$. Then
\begin{equation}\label{eq:general-moment-convergence}
 \E\left[\left|\frac{\cD_{N_m}-M_m}{\sqrt{V_m}}\right|^p\right]
 \longrightarrow\E[|G|^p],\qquad G\sim\Normal(0,1).
\end{equation}
If $p$ is a positive integer, the corresponding ordinary moments converge
as well.
\end{proposition}

\begin{proof}
Maximally couple $N\bmod B_r$ to a uniform residue.  The argument in
\Cref{lem:clock-coupling} uses only the probability that the first $r$
clock digits fail to agree; it does not use the law of the upper digits of
$N$.  With this failure probability now equal to $\delta_{N,r}$, the same
level-by-level construction and \eqref{eq:v71-auto-0057} give
\begin{align}
 \|D_N(Y)-\mathcal P_m\|_p
 &\leq C_p+2\sum_{j=r}^m b_j
  +2\left(\sum_{j=1}^m b_j\right)\delta_{N,r}^{1/p}\notag\\
 &\leq C_p\sqrt{S_m}\,\rho_{N,r}^{(p)},
 \qquad p\geq2.\label{eq:general-coupling-Lp}
\end{align}
Indeed, \eqref{eq:v71-auto-0054} defines
$C_p=\sum_{k=1}^{\infty}2^{-k/p}$, and for $p\geq2$ one has
$C_p\geq C_2=1+\sqrt2>2$.
Here the digital decomposition also covers $N=0$, since then all clock
digits and all local terms vanish.

Now take $p=2$ in \eqref{eq:general-coupling-Lp}.  By
\eqref{eq:product-exact-mean} and Cauchy--Schwarz,
\begin{align}
 |\E[\cD_N]-M_m|
 &=|\E[\cD_N-\mathcal P_m]|\leq\|\cD_N-\mathcal P_m\|_2
 \leq C_2\sqrt{S_m}\,\rho_{N,r}^{(2)}
 ,\label{eq:v71-auto-0220}
\end{align}
which is \eqref{eq:general-mean}.  

For the standard deviations, write
$\sigma_N:=\|\cD_N-\E[\cD_N]\|_2$ and
$\sqrt{V_m}=\|\mathcal P_m-M_m\|_2$ by
\eqref{eq:product-exact-var}.  The reverse triangle inequality in
$L^2$ and one more use of Cauchy--Schwarz give
\begin{align}
 |\sigma_N-\sqrt{V_m}|
 &\leq\|(\cD_N-\E[\cD_N])-(\mathcal P_m-M_m)\|_2\label{eq:v71-auto-0221}\\
 &\leq\|\cD_N-\mathcal P_m\|_2
       +|\E[\cD_N-\mathcal P_m]|\leq2C_2\sqrt{S_m}\,\rho_{N,r}^{(2)}
 .\label{eq:v71-auto-0223}
\end{align}
This proves \eqref{eq:general-sd}.

For the distributional  \eqref{eq:general-rate}, set
\begin{equation}\label{eq:v71-auto-0058}
 X:=\frac{\mathcal P_m-M_m}{\sqrt{V_m}},
 \qquad
 Y_N:=\frac{\cD_N-M_m}{\sqrt{V_m}}.
\end{equation}
The variance comparison \eqref{eq:variance-comparison} and
\eqref{eq:general-coupling-Lp} imply
\begin{equation}\label{eq:general-normalized-L2}
 \|X-Y_N\|_2
 \leq\sqrt{240}\,C_2\rho_{N,r}^{(2)}.
\end{equation}
For any $\eta>0$, the inclusions
\begin{equation}\label{eq:v71-auto-0059}
 \{X\leq t-\eta\}\setminus\{|X-Y_N|>\eta\}
 \subseteq\{Y_N\leq t\}
 \subseteq\{X\leq t+\eta\}\cup\{|X-Y_N|>\eta\}
\end{equation}
together with the bound $\sup_x\Phi'(x)=1/\sqrt{2\pi}$ and Chebyshev's
inequality, yield
\begin{equation}\label{eq:v71-auto-0060}
 \dK(Y_N,\Normal(0,1))
 \leq\dK(X,\Normal(0,1))+\frac{\eta}{\sqrt{2\pi}}
   +\frac{240C_2^2(\rho_{N,r}^{(2)})^2}{\eta^2}.
\end{equation}
This event-inclusion argument is a quantitative one-dimensional analogue,
in the present Gaussian setting, of a standard theorem: given two sequences of random variables on a metric space, convergence in law of the first sequence plus convergence to zero in probability of the difference of the two sequences imply convergence in law of the second sequence; cf.~\cite[Theorem~3.1, p.~27]{Billingsley}.
Since $\rho_{N,r}^{(2)}>0$, choose
$\eta=(\rho_{N,r}^{(2)})^{2/3}$ in \eqref{eq:v71-auto-0060}.  
We obtain the explicit transfer estimate
\begin{equation}\label{eq:general-explicit-transfer-rate}
 \dK(Y_N,\Normal(0,1))
 \leq\dK(X,\Normal(0,1))
 +\left(\frac1{\sqrt{2\pi}}+240C_2^2\right)
   (\rho_{N,r}^{(2)})^{2/3}.
\end{equation}
Combining this with \eqref{eq:product-rate} of \Cref{thm:product-clt} proves
\eqref{eq:general-rate}.

Finally, take $N=N_m$ and $r=r(m)$ along a sequence satisfying the moment
assumptions in the statement.  Since
$0\leq\delta_{N,r}\leq1$ and $q>2$, we have
$\rho_{N,r}^{(2)}\leq\rho_{N,r}^{(q)}\to0$ along the sequence.  Thus
\eqref{eq:general-rate} gives convergence in distribution to the standard
normal.  On the other hand, \eqref{eq:general-coupling-Lp} with exponent
$q$ and \eqref{eq:variance-comparison} give
\begin{align}
 \left\|\frac{\cD_N-\mathcal P_m}{\sqrt{V_m}}\right\|_q
 &\leq\sqrt{240}\,C_q\rho_{N,r}^{(q)},
 \label{eq:v71-auto-0224}
\end{align}
with $C_q$ as  in
\eqref{eq:v71-auto-0054}.  Now set
\begin{equation*}
 A_q:=\sup_m\left\|\frac{\mathcal P_m-M_m}{\sqrt{V_m}}\right\|_q,
 \qquad
 R_q:=\sup_m\rho_{N_m,r(m)}^{(q)}.
\end{equation*}
Then $A_q<\infty$ by \eqref{eq:v71-auto-0051}, while $R_q<\infty$ because
$\rho_{N_m,r(m)}^{(q)}\to0$.  Minkowski's inequality and
\eqref{eq:v71-auto-0224} therefore give the concrete bound
\begin{equation}
 \sup_m\E\left[\left|\frac{\cD_N-M_m}{\sqrt{V_m}}\right|^q\right]
 \leq\left(A_q+\sqrt{240}\,C_qR_q\right)^q<\infty.
 \label{eq:v71-auto-0225}
\end{equation}
Because $q>p$, the $p$th absolute powers are uniformly integrable.
Convergence in distribution therefore gives convergence of the $p$th
absolute moments. When $p$ is an integer, the same argument applied to
$x\mapsto x^p$ gives convergence of the ordinary moments as well.
\end{proof}

\begin{remark}\label{rem:block-uniform-time}
For example, let $N=R_r+B_rK$, where $R_r$ is uniform on
$\{0,\ldots,B_r-1\}$ and independent of an otherwise arbitrary
integer-valued random variable $K$, subject
only to $0\leq N<B_m$.  Then $\delta_{N,r}=0$ and hence $\rho^{(p)}_{N,r}=\frac{1+\sum_{j=r}^m b_j}{\sqrt{S_m}}$. 
Therefore, for the comparison of \(\mathcal D_N\) with the Haar-clock model \(\mathcal P_m\), the law of the upper clock digits (with indices greater than $r$) is irrelevant, provided their total possible contribution \(\sum_{j=r}^m b_j\) is negligible relative to \(\sqrt{S_m}\).
In other words, along a sequence \(m\to\infty\), with \(r=r(m)\le m\), the upper digits may have an arbitrary law, as far as the clock transfer is concerned, provided
$\sum_{j=r(m)}^m b_j=o(\sqrt{S_m})$. In fact, this assumption implies $\rho^{(p)}_{N,r(m)}\to 0$ and the latter, together with   \eqref{eq:v71-auto-0224} and \eqref{eq:variance-comparison}, yields $\left\|\frac{\cD_N-\mathcal P_m}{\sqrt{V_m}}\right\|_p\to 0$.
\end{remark}
\begin{remark}\label{rk:H=B_r-sequence}
Suppose that $H$ tends to infinity along the sparse sequence
$H_k=B_k$, $k\geq1$.
Then $m(H_k)=k$ and $N_{H_k}$ is uniform on
$\{0,1,\ldots,B_k-1\}$.  Taking $m=r=k$ in
\Cref{rem:block-uniform-time}, we have
$\delta_{N_{H_k},k}=0$ and
\[
 \rho_{N_{H_k},k}^{(p)}=\frac{1+b_k}{\sqrt{S_k}}.
\]
If $\Delta_k\to0$, then
$b_k/\sqrt{S_k}\leq\Delta_k^{1/4}\to0$.  Thus the product-model CLT
in \Cref{thm:product-clt} transfers to give
\[
 \frac{\cD_{H_k}-M_k}{\sqrt{V_k}}
 \Longrightarrow \Normal(0,1)
 \qquad(k\to\infty).
\]
\end{remark}

Now, to prove \Cref{thm:main} we use \Cref{prop:general-time} with $N=N_H$ and leverage the  assumptions
 \eqref{eq:main-conditions}.

\begin{proof}[Proof of \Cref{thm:main}]
Put $m=m(H)$ and apply \Cref{prop:general-time} with $N=N_H$.
By \eqref{H-sandwiched}, this random variable is supported on
$\{0,\ldots,B_m-1\}$, and it is independent of $Y$.  Equations
\eqref{eq:v71-auto-0015} and \eqref{eq:time-prefix-TV}, together with the
definition \eqref{eq:general-delta}, give
$\delta_{N_H,r}=\eps_{H,r}$.  Hence the definitions
\eqref{eq:general-rho} and \eqref{eq:rho-p} give
$\rho_{N_H,r}^{(p)}=\rho_{H,r,m}^{(p)}$ for every $p\geq1$.
Consequently, \eqref{eq:general-mean}--\eqref{eq:general-rate} give
\eqref{eq:mean-bound}--\eqref{eq:main-rate}.  Since the implied constants
in \Cref{prop:general-time} are absolute and uniform, taking their maximum
gives the single constant $C$ in the statement.

Now suppose that \eqref{eq:main-conditions} holds.  For each $H$, choose
an index $r(H)$ attaining the minimum of
$\rho_{H,r,m(H)}^{(2)}$ over $1\leq r\leq m(H)$.  This defect tends to
zero, so \eqref{eq:mean-bound} gives \eqref{eq:main-mean}.  Likewise,
\eqref{eq:sd-bound} gives
$|\sqrt{\Var(\cD_H)}-\sqrt{V_m}|=o(\sqrt{S_m})$, while
\eqref{eq:variance-comparison} implies
$\sqrt{\Var(\cD_H)}+\sqrt{V_m}=O(\sqrt{S_m})$.  Multiplying these two
estimates proves \eqref{eq:main-var}.  Finally, the bound \eqref{eq:main-rate}, the assumption
$\Delta_m\to0$, and the chosen defect tending to zero prove
\eqref{eq:main-explicit}.
\end{proof}

\begin{proof}[Proof of \Cref{thm:moments}]
Let $m=m(H)$ and take the indices $r=r(H)$ from the hypothesis.  By
\Cref{lem:clock-coupling} and \eqref{eq:variance-comparison}, we have the limit
\begin{equation}\label{eq:v71-auto-0066}
 \left\|\frac{\cD_H-\mathcal P_m}{\sqrt{V_m}}\right\|_q
 \leq C_q\rho_{H,r,m}^{(q)}\longrightarrow0.
\end{equation}
The Haar-model moment bound \eqref{eq:v71-auto-0051} shows that
\begin{equation}\label{eq:v71-auto-0067}
 \sup_H\E\left[\left|\frac{\mathcal P_m-M_m}{\sqrt{V_m}}\right|^q\right]<\infty.
\end{equation}
By Minkowski's inequality, the normalized random-time statistics
\begin{equation}\label{eq:v71-auto-0068}
 Z_H:=\frac{\cD_H-M_m}{\sqrt{V_m}}
\end{equation}
are therefore uniformly bounded in $L^q$.  The assumptions
\eqref{eq:main-conditions} and \Cref{thm:main} give
$Z_H\Longrightarrow G$, where $G\sim\Normal(0,1)$.  Since $q>p$, the
family $(|Z_H|^p)_H$ is uniformly integrable.  It follows that
\begin{equation}\label{eq:v71-auto-0069}
 \E[|Z_H|^p]\longrightarrow\E[|G|^p],
\end{equation}
which is \eqref{eq:moment-convergence}.  If $p$ is a positive integer,
then $Z_H^p\Longrightarrow G^p$ and
$|Z_H^p|=|Z_H|^p$ is uniformly integrable.  Hence
$\E[Z_H^p]\to\E[G^p]$, proving
\eqref{eq:ordinary-moment-convergence}.
\end{proof}

\begin{proof}[Proof of \Cref{cor:old-condition}]
Put $m=m(H)$.  It is enough to consider sufficiently large $m$, so that all the indices
introduced below are positive.  Let
$L=\lceil16\log_2m\rceil$ and $r=m-L$.  Since
$B_{m-1}<H\leq B_m$, \Cref{lem:time-digits-Haar} and $b_j\geq2$ give
\begin{equation}\label{eq:old-epsilon}
 \eps_{H,r}\leq\frac{B_r}{H}
 <\frac{B_r}{B_{m-1}}
 =\frac1{b_{r+1}\cdots b_{m-1}}
 \leq2^{-(L-1)}.
\end{equation}
Note that there are $L+1$ indices in the range $r\leq j\leq m$.  H\"older's
inequality (with conjugate exponents $4$ and $4/3$) gives
\begin{equation}\label{eq:old-top}
 \frac{\sum_{j=r}^m b_j}{\sqrt{S_m}}
 \leq (L+1)^{3/4}
       \frac{(\sum_{j=r}^m b_j^4)^{1/4}}{\sqrt{S_m}}
 \leq (L+1)^{3/4}\Delta_m^{1/4}=O((\log m)^{3/4}\Delta_m^{1/4}).
\end{equation}
The hypothesis $\Delta_m=O(m^{-1/2-\eta})$ and
\eqref{eq:old-top} therefore give
\[
 \frac{\sum_{j=r}^m b_j}{\sqrt{S_m}}
 =O((\log m)^{3/4}m^{-1/8-\eta/4}).
\]
On the other hand, by Cauchy--Schwarz,
\begin{equation}\label{eq:old-all-levels}
 \frac{\sum_{j=1}^m b_j}{\sqrt{S_m}}\leq\sqrt m.
\end{equation}
Combining \eqref{eq:old-epsilon}--\eqref{eq:old-all-levels} with the
definition \eqref{eq:rho-p}, and using
$2^{-(L-1)/2}=O(m^{-8})$ and $S_m\geq4m$, yields
\begin{equation}\label{eq:old-rho}
 \rho_{H,r,m}^{(2)}
 =O(m^{-1/2})+O((\log m)^{3/4}m^{-1/8-\eta/4})+O(m^{-15/2}).
\end{equation}
The hypothesis also gives $\Delta_m\to0$, while \eqref{eq:old-rho} tends
to zero as $m\to\infty$.  Thus the two conditions in
\eqref{eq:main-conditions} hold uniformly in $H$.

We note that the hypothesis can be nonvacuous only for
$\eta\leq1/2$: indeed, Cauchy--Schwarz gives
$\Delta_m\geq1/m$.  We may therefore use $0<\eta\leq1/2$ when comparing
the three terms in \eqref{eq:old-rho}. Since $\tfrac{1}{8}+\tfrac{\eta}{4}\le\tfrac{1}{4}$, the second term in \eqref{eq:old-rho} dominates the other two. We obtain
\begin{equation}\label{eq:v71-auto-0094}
 (\rho_{H,r,m}^{(2)})^{2/3}
 =O((\log m)^{1/2}m^{-1/12-\eta/6}),
\end{equation}
while
\begin{equation}\label{eq:v71-auto-0095}
 \Delta_m^{1/5}=O(m^{-1/10-\eta/5})
 =O(m^{-1/12-\eta/6}).
\end{equation}
Substitution in \eqref{eq:main-rate} proves \eqref{eq:old-rate}.

Finally, fix a moment order $p>0$ and choose
$q>\max\{p,2\}$.  Repeat the argument with
$L=\lceil q\log_2m\rceil$.  The upper-level term in
\eqref{eq:old-top} still tends to zero, and the coupling-failure term is now
bounded by
\begin{equation}\label{eq:v71-auto-0096}
 \sqrt m\,2^{-(L-1)/q}=o(1).
\end{equation}
Hence $\rho_{H,r,m}^{(q)}\to0$, uniformly in $H$, and
\Cref{thm:moments} proves the stated absolute and ordinary moment
convergences.
\end{proof}

\begin{proof}[Proof of \Cref{cor:regular-radices}]
Put $m=m(H)$.
The hypothesis includes every bounded radix sequence: if
$2\leq b_j\leq K$, take $\alpha=0$, $c_-=2$, and $c_+=K$.
For every $\beta\geq0$, integral comparison and the bound $j\leq m$ give
\begin{equation}\label{eq:regular-power-sum}
 \frac{m^{\beta+1}}{\beta+1}
 \leq\sum_{j=1}^m j^\beta\leq m^{\beta+1}.
\end{equation}
Consequently, the lower bound in \eqref{eq:regular-radix-bounds}, with
$\beta=2\alpha$, and the upper bound, with $\beta=4\alpha$, give
\begin{equation}\label{eq:v71-auto-0097}
 S_m\geq\frac{c_-^2}{2\alpha+1}m^{2\alpha+1},
 \qquad
 Q_m\leq c_+^4m^{4\alpha+1}.
\end{equation}
An upper bound for $S_m$ and a
lower bound for $Q_m$ are not needed.  Define
\begin{equation}\label{eq:regular-constants}
 A_0:=\frac{\sqrt{2\alpha+1}}{c_-},\qquad
 R_0:=\frac{c_+}{c_-}\sqrt{2\alpha+1},\qquad
 D_0:=\left(\frac{c_+}{c_-}\right)^4(2\alpha+1)^2.
\end{equation}
Cauchy--Schwarz gives $S_m^2\leq mQ_m$, so \eqref{eq:v71-auto-0097}
yields the explicit bounds
\begin{equation}\label{eq:regular-delta}
 \frac1m\leq\Delta_m\leq\frac{D_0}{m}.
\end{equation}
For $m\geq8$, let $L:=\lceil2\log_2m\rceil+1$ and $r:=m-L$.
Then $1\leq r\leq m$.  As in \eqref{eq:old-epsilon}, the standing
assumption $b_j\geq2$ gives, uniformly for $B_{m-1}<H\leq B_m$,
\begin{equation}\label{eq:v71-auto-0098}
 \eps_{H,r}\leq2^{-(L-1)}\leq m^{-2}.
\end{equation}
There are $L+1$ indices in the range $r\leq j\leq m$.  The upper radix
bound in \eqref{eq:regular-radix-bounds} and the lower bound for $S_m$ in
\eqref{eq:v71-auto-0097} therefore give
\begin{equation}\label{eq:v71-auto-0099}
 \frac{\sum_{j=r}^m b_j}{\sqrt{S_m}}
 \leq R_0\frac{L+1}{\sqrt m}.
\end{equation}
Moreover,
\begin{equation}\label{eq:regular-remaining-terms}
 \frac1{\sqrt{S_m}}\leq\frac{A_0}{\sqrt m},
 \qquad
 \frac{\sum_{j=1}^m b_j}{\sqrt{S_m}}\leq\sqrt m,
\end{equation}
where the second inequality is Cauchy--Schwarz.  Substitution of
\eqref{eq:v71-auto-0098}--\eqref{eq:regular-remaining-terms} into
\eqref{eq:rho-p}, followed by $L+1\leq2\log_2m+3$, gives
\begin{equation}\label{eq:regular-rho}
 \rho_{H,r,m}^{(2)}
 \leq\frac{A_0+1+R_0(2\log_2m+3)}{\sqrt m}.
\end{equation}
In particular, this tends to zero uniformly in $H$, and
\eqref{eq:regular-delta} verifies the first condition in
\eqref{eq:main-conditions}.

Let $C_{\mathrm{CLT}}$ be the absolute constant in
\eqref{eq:main-rate}.  Equations \eqref{eq:regular-delta} and
\eqref{eq:regular-rho} give, for $m\geq8$,
\begin{align}
 \dK\left(\frac{\cD_H-M_m}{\sqrt{V_m}},\Normal(0,1)\right)
 \leq C_{\mathrm{CLT}}\Bigl(&D_0^{1/5}m^{-1/5}\notag\\
 &+\left(A_0+1+R_0(2\log_2m+3)\right)^{2/3}m^{-1/3}\Bigr).
 \label{eq:v71-auto-0100}
\end{align}
The elementary inequality
$\log_2m\leq5m^{1/5}/(e\log2)$ for $m\geq1$ shows that, with
\begin{equation}\label{eq:regular-rate-constant}
 E_0:=A_0+1+R_0\left(3+\frac{10}{e\log2}\right),
 \qquad
 C_{\mathrm{reg}}:=\max\left\{7^{1/5},
 C_{\mathrm{CLT}}\bigl(D_0^{1/5}+E_0^{2/3}\bigr)\right\},
\end{equation}
the right-hand side of \eqref{eq:v71-auto-0100} is at most
$C_{\mathrm{reg}}m^{-1/5}$.  For $m\leq7$, the same conclusion follows
from the trivial bound $\dK\leq1$.  This proves \eqref{eq:regular-rate}.

Finally, fix a moment order $p>0$, set $q=p+2$, and, for all sufficiently
large $m$, take
\begin{equation}\label{eq:v71-auto-0101}
 L_q:=\lceil q\log_2m\rceil+1,
 \qquad r_q:=m-L_q.
\end{equation}
Then \eqref{eq:v71-auto-0098} with $L_q,r_q$ in place of $L,r$ gives
$\eps_{H,r_q}\leq m^{-q}$.  Repeating
\eqref{eq:v71-auto-0099}--\eqref{eq:regular-rho} yields
\begin{equation}\label{eq:regular-rho-q}
 \rho_{H,r_q,m}^{(q)}
 \leq\frac{A_0+1+R_0(q\log_2m+3)}{\sqrt m}
 \longrightarrow0
\end{equation}
uniformly in $H$.  Since $q>\max\{p,2\}$, the moment assertions follow
from \Cref{thm:moments}.
\end{proof}

\section{Irregular radix growth: a Gaussian example and a non-Gaussian obstruction}
\label{sec:irregular-radices}

This section uses two explicit unbounded and irregular radix sequences to
show both the scope and the limitations of our results.
We first study a sequence of radices that differs from a constant sequence only at square indices.  It violates the assumption 
\eqref{eq:old-condition} of \Cref{cor:old-condition} 
and is not covered by
\Cref{cor:regular-radices}, but we can still use  \Cref{thm:main} to obtain quantitative Gaussian
limit uniformly through all time horizons.  We then construct a sparse
dominant-radix sequence satisfying $\sum_{j\leq m}b_j=O(m)$ for which a
single level controls the fluctuations along a subsequence, producing an
explicit nontrivial non-Gaussian limit.  The second example shows why some
Feller-type condition is indispensable for a general Gaussian theorem.

\begin{example}\label{ex:square-spikes}
Let
\begin{equation}\label{eq:square-spike-radices}
 b_{k^2}=k+1\quad(k\geq1),\qquad
 b_j=2\quad\text{when $j$ is not a square}.
\end{equation}
Then, uniformly for $H\geq2$, with $m=m(H)$,
\begin{equation}\label{eq:square-spike-rate}
 \dK\left(
 \frac{\cD_H-M_m}{\sqrt{V_m}}\:,\:\Normal(0,1)\right)
 =O\left(m^{-1/10}\right).
\end{equation}
Moreover, the absolute moments of every positive order, and the ordinary
moments of every positive integer order, converge to the corresponding
Gaussian moments.  The deterministic normalization may also be simplified:
\begin{equation}\label{eq:square-spike-simple-rate}
 \dK\left(
 \frac{\cD_H-m/6}{m^{3/4}/\sqrt{720}}\:,\:\Normal(0,1)\right)
 =O\left(m^{-1/10}\right).
\end{equation}
The same moment conclusions hold with this normalization.  Since
$m(H)\sim\log_2H$, both rates are $O((\log H)^{-1/10})$.
\end{example}

\begin{proof}
Put $n=\lfloor\sqrt m\rfloor$, and write
\begin{equation}\label{eq:square-spike-harmonic-sums}
 h_n^{(a)}:=\sum_{\ell=1}^n\frac1{\ell^a}
 \qquad(a>0).
\end{equation}
Direct summation gives
\begin{align}
 S_m
 &=4m+\sum_{k=1}^n\bigl((k+1)^2-4\bigr)
   =4m+\frac{n(2n^2+9n-11)}6,
   \label{eq:square-spike-S}\\
 Q_m
 &=16m+\sum_{k=1}^n\bigl((k+1)^4-16\bigr)=16m+
 \frac{6n^5+45n^4+130n^3+180n^2-361n}{30}.
 \label{eq:square-spike-Q}
\end{align}
Since $n=\sqrt m+O(1)$, it follows that
\begin{equation}\label{eq:square-spike-delta}
 S_m=\frac13m^{3/2}+O(m),\qquad
 Q_m=\frac15m^{5/2}+O(m^2),\qquad
 \Delta_m=\frac9{5\sqrt m}+O(m^{-1}).
\end{equation}
The tower heights and normalization constants can also be computed
explicitly:
\begin{align}
 B_m&=2^{m-n}(n+1)!,\label{eq:square-spike-B}\\
 M_m&=\frac m8+\frac{n^2}{24}
       -\frac{h_{n+1}^{(1)}-1}{12},\label{eq:square-spike-M}\\
 V_m&=\frac{5m}{192}
 +\frac{n^3+\frac92n^2-\frac94n
       -13(h_{n+1}^{(2)}-1)}{720}
 +\frac{1-B_m^{-1}}{12}.
 \label{eq:square-spike-V}
\end{align}
In particular,
\begin{equation}\label{eq:square-spike-MV-asymptotic}
 M_m=\frac m6+O(\sqrt m),\qquad
 V_m=\frac1{720}m^{3/2}+O(m).
\end{equation}
Furthermore, Stirling's formula and \eqref{eq:square-spike-B} give
\begin{equation}\label{eq:square-spike-depth}
 \log B_m=m\log2+O(\sqrt m\log m),
\end{equation}
which implies $m(H)\sim\log_2H$.

We next estimate the time defect.  For all sufficiently large $m$, let
\begin{equation}\label{eq:square-spike-truncation}
 L:=\lceil2\log_2m\rceil+1,\qquad r:=m-L.
\end{equation}
Uniformly for $B_{m-1}<H\leq B_m$, the argument in
\eqref{eq:v71-auto-0098} gives
\begin{equation}\label{eq:square-spike-epsilon}
 \eps_{H,r}\leq2^{-(L-1)}\leq m^{-2}.
\end{equation}
Because consecutive squares near $m$ are separated by order $\sqrt m$, the
interval $[r,m]$, whose length is $O(\log m)$, contains at most one square
when $m$ is sufficiently large.  Hence
\begin{equation}\label{eq:square-spike-top}
 \sum_{j=r}^m b_j
 \leq2(L+1)+n+1=O(\sqrt m).
\end{equation}
On the other hand,
\begin{equation}\label{eq:square-spike-all-levels}
 \sum_{j=1}^m b_j
 =2m+\sum_{k=1}^n(k-1)
 =2m+\frac{n(n-1)}2=O(m).
\end{equation}
Equations \eqref{eq:square-spike-S},
\eqref{eq:square-spike-epsilon}, \eqref{eq:square-spike-top}, and
\eqref{eq:square-spike-all-levels} therefore yield
\begin{equation}\label{eq:square-spike-rho}
 \rho_{H,r,m}^{(2)}
 =\frac{1+O(\sqrt m)+O(m)m^{-1}}
        {(\frac1{3}m^{3/2}+O(m))^{1/2}}
 =O(m^{-1/4}),
\end{equation}
uniformly in $H$.  Substitution of \eqref{eq:square-spike-delta} and
\eqref{eq:square-spike-rho} into \eqref{eq:main-rate} gives
\begin{equation}\label{eq:square-spike-rate-proof}
 \dK\left(
 \frac{\cD_H-M_m}{\sqrt{V_m}}\:,\:\Normal(0,1)\right)
 =O\left(m^{-1/10}+m^{-1/6}\right)
 =O(m^{-1/10}),
\end{equation}
proving \eqref{eq:square-spike-rate}; the finitely many remaining values of
$m$ are absorbed by enlarging the implied constant.

To prove moment convergence, fix $p>0$, choose $q>\max\{p,2\}$, and instead
take
\begin{equation}\label{eq:square-spike-moment-truncation}
 L_q:=\lceil q\log_2m\rceil+1,\qquad r_q:=m-L_q.
\end{equation}
Then $\eps_{H,r_q}\leq m^{-q}$.  Since $L_q=O_q(\log m)$, the same argument
as above gives
\begin{equation}\label{eq:square-spike-rho-q}
 \rho_{H,r_q,m}^{(q)}=O_q(m^{-1/4})\longrightarrow0
\end{equation}
uniformly in $H$.  The moment assertions consequently follow from
\Cref{thm:moments}.

Finally, \eqref{eq:square-spike-MV-asymptotic} gives
\begin{equation}\label{eq:square-spike-affine-errors}
 \frac{M_m-m/6}{m^{3/4}/\sqrt{720}}=O(m^{-1/4}),
 \qquad
 \frac{\sqrt{V_m}}{m^{3/4}/\sqrt{720}}
 =1+O(m^{-1/2}).
\end{equation}
An affine change of normalization in \eqref{eq:square-spike-rate} therefore
proves \eqref{eq:square-spike-simple-rate}, and preserves the asserted moment
convergence.
\end{proof}

\begin{remark}\label{rem:square-spikes-literature}
The estimate \eqref{eq:square-spike-delta} shows that
\eqref{eq:old-condition} fails for every $\eta>0$. 
The example is also outside the scope of
\Cref{cor:regular-radices}: its radices are unbounded along the squares but
equal to $2$ along infinitely many indices.  
\end{remark}

We now turn to the opposite regime, in which a sparse radix dominates the
quadratic scale.  The first condition in \eqref{eq:main-conditions} prevents
any single radix from contributing a non-negligible fraction of the quadratic
scale \(S_m\), asymptotically as $m\to\infty$. Indeed, by
\Cref{lem:feller-equivalence}, it is equivalent to the Feller-type condition
\begin{equation*}
 \lim_{m\to\infty}\max_{1\leq j\leq m}\frac{b_j^2}{S_m}=0.
\end{equation*}
The stronger sufficient condition in \Cref{cor:old-condition} implies this
condition.  The next example lies at the opposite extreme along a sparse
sequence of levels $m=n_k$.  Along those levels,
$S_m\sim b_m^2$ and $Q_m\sim b_m^4$, so $\Delta_m\to1$.  Moreover, at
$H=B_m$ one has $\eps_{H,r}=0$, but for every $r\leq m$,
\begin{equation*}
 \rho_{B_m,r,m}^{(2)}\geq\frac{b_m}{\sqrt{S_m}}\longrightarrow1
\end{equation*}
along $m=n_k$. Thus both assumptions in \eqref{eq:main-conditions} fail.  The resulting
non-Gaussian limit shows that some anti-dominance hypothesis is genuinely
needed for a Gaussian theorem covering arbitrary radix sequences. We do not claim 
not assert that our particular sufficient conditions are necessary in every
individual case.  Crucially, the example still satisfies
$\sum_{j\leq m}b_j=O(m)$, and therefore disproves the mixed-radix assertion
in \cite[Remark~1.2]{LevinMerzbach} under its stated hypothesis alone.

\begin{example}
\label{ex:dominant-radix}
For $k\ge1$, let  $n_k:=2^{2^k}$. Define
\begin{equation}\label{eq:dominant-radices}
 b_j:=
 \begin{cases}
  n_k,&j=n_k\text{ for some }k\geq1,\\
  2,&\text{otherwise}.
 \end{cases}
\end{equation}
Then
\begin{equation}\label{eq:dominant-average-bound}
 \sum_{j\leq m}b_j\leq4m
\end{equation}
for every $m\geq1$.  Nevertheless, along $m=n_k$ the Feller condition
fails maximally: in fact,
$S_m\sim b_m^2$, $Q_m\sim b_m^4$, and hence $\Delta_m\to1$.  Along the
complete tower heights $H_k=B_{n_k}$,
\begin{equation}\label{eq:dominant-limit}
 \frac{\cD_{H_k}-\E[\cD_{H_k}]}
      {\sqrt{\Var(\cD_{H_k})}}
 \Longrightarrow
 \sqrt{240}\left(\min(R,S)-RS-\frac1{12}\right),
\end{equation}
where $R,S$ are independent uniform random variables.  The limit is bounded
and nondegenerate, hence is not Gaussian.
\end{example}

\begin{proof}
For $m<n_1$, it is immediate to verify \eqref{eq:dominant-average-bound}.  For $n_k\leq m<n_{k+1}$, the rapid growth of $(n_k)$ gives
$\sum_{\ell\leq k}n_\ell\leq2n_k\leq2m$.  Hence
\begin{equation}\label{eq:dominant-average-proof}
 \sum_{j\leq m}b_j
 =2m+\sum_{n_\ell\leq m}(n_\ell-2)
 \leq4m,
\end{equation}
which proves \eqref{eq:dominant-average-bound}.

We now restrict to $m=n_k$.  Since $n_\ell^2=n_{\ell+1}$,
\begin{align}
 S_{m-1}
 &=4(m-1)+\sum_{\ell<k}(n_\ell^2-4)=O(m),
 \label{eq:dominant-lower-S}\\
 Q_{m-1}
 &=16(m-1)+\sum_{\ell<k}(n_\ell^4-16)=O(m^2).
 \label{eq:dominant-lower-Q}
\end{align}
As $b_m=m$, these estimates imply
$S_m\sim b_m^2$, $Q_m\sim b_m^4$, and $\Delta_m\to1$ along this
subsequence.

Let $e^{(m)}:=c(N_{B_m})$, let
$a(Y)=(a_i(Y))_{i\geq0}$ be the endpoint digit string of $Y$.  Using \eqref{eq:xi-original}, we  set
\begin{equation}\label{eq:v71-auto-0105}
 \xi_i^{(m)}:=\xi_i(e^{(m)},Y).
\end{equation}
Since $N_{B_m}<B_m$, the clock digits above level $m-1$ vanish.  Therefore the
digital decomposition contains only the levels $0,\ldots,m-1$.  Put
$\mathcal R_m:=\sum_{i=0}^{m-2}\xi_i^{(m)}$.

The digits $e_0^{(m)},\ldots,e_{m-1}^{(m)}$ are independent and uniform.
Append independent uniform digits in all positions $i\geq m$ to obtain a
Haar clock string $e^*$.  Using the same endpoint string $a(Y)$, let
$\Xi_i^*$ be the corresponding product-model contributions.  Since the two
clock strings agree through level $m-1$, the argument in
\eqref{eq:coupling-low-level-Lp}--\eqref{eq:v71-auto-0054}, with $r=m$ and
$p=2$, gives
\begin{equation}\label{eq:dominant-low-level-coupling}
 \left\|\mathcal R_m-\sum_{i=0}^{m-2}\Xi_i^*\right\|_2\leq C_2.
\end{equation}
The sum on the right is a copy of $\mathcal P_{m-1}$.  Hence
\Cref{prop:exact-covariance}, \eqref{eq:variance-comparison}, and
\eqref{eq:dominant-lower-S} yield the  bound
\begin{equation}\label{eq:dominant-lower-levels}
 \|\mathcal R_m-\E[\mathcal R_m]\|_2
 \leq\sqrt{V_{m-1}}+2C_2
 \ll\sqrt{S_{m-1}}+1
 =O(\sqrt m)=o(b_m).
\end{equation}

At level $m-1$, write
$A=a_{m-1}(Y)$ and $E=e_{m-1}^{(m)}=c_{m-1}(N_{B_m})$.  These are independent uniform
digits in $\{0,\ldots,b_m-1\}$.  The finite time tail is
$U_m(e^{(m)})=0$, whereas the endpoint tail $U_m(a(Y))$ is uniform and
hence positive almost surely.  The local normal form consequently gives
\begin{equation}\label{eq:v71-auto-0106}
 \frac{\xi_{m-1}^{(m)}}{b_m}
 =\min\left(\frac A{b_m},\frac E{b_m}\right)
  -\frac A{b_m}\frac E{b_m}
  +\frac{\1_{\{A<E\}}}{b_m}-\frac{E\,U_m(a(Y))}{b_m^2}.
\end{equation}
The last two terms have absolute value at most $2/b_m$. 
Take independent random variables \(R,S\), each uniformly distributed on
\([0,1)\), and, for every spike index \(m=n_k\), define
\[
 A_m=\lfloor b_mR\rfloor,
 \qquad
 E_m=\lfloor b_mS\rfloor.
\]
Then \(A_m\) and \(E_m\) are independent and uniform on
\(\{0,\ldots,b_m-1\}\).  Moreover,
\[
 \left|\frac{A_m}{b_m}-R\right|<\frac1{b_m},
 \qquad
 \left|\frac{E_m}{b_m}-S\right|<\frac1{b_m}.
\]
Since \(b_m=b_{n_k}=n_k\to\infty\) along the spike subsequence,
\(A_m/b_m\to R\) and \(E_m/b_m\to S\) almost surely.
It follows, for every fixed $p>0$, that
\begin{equation}\label{eq:dominant-Lp}
 \frac{\xi_{m-1}^{(m)}}{b_m}\longrightarrow
 L:=\min(R,S)-RS
 \qquad\text{in }L^p.
\end{equation}

We record the first two moments of $L$.  By symmetry across the diagonal of the
unit square,
\begin{align}
 \E[L]
 &=2\int_0^1\int_u^1u(1-v)\,dv\,du
 =\int_0^1u(1-u)^2\,du=\frac1{12},\label{eq:v71-auto-0240}\\
 \E[L^2]
 &=2\int_0^1\int_u^1u^2(1-v)^2\,dv\,du
 =\frac23\int_0^1u^2(1-u)^3\,du=\frac1{90}.\label{eq:v71-auto-0241}
\end{align}
Thus $\Var(L)=1/90-1/144=1/240>0$.  Taking $p=2$ in
\eqref{eq:dominant-Lp} and using \eqref{eq:dominant-lower-levels} gives
\begin{equation}\label{eq:dominant-centered-L2}
 \frac{\cD_{B_m}-\E[\cD_{B_m}]}{b_m}
 \longrightarrow L-\frac1{12}
 \qquad\text{in }L^2
\end{equation}
along $m=n_k$.  In particular,
$\Var(\cD_{B_m})/b_m^2\to1/240$.  Consequently,
\begin{equation}\label{eq:v71-auto-0107}
 \frac{\cD_{B_m}-\E[\cD_{B_m}]}
      {\sqrt{\Var(\cD_{B_m})}}
 \Longrightarrow\frac{L-1/12}{\sqrt{1/240}},
\end{equation}
which is \eqref{eq:dominant-limit}.  Since this limit is bounded and has
positive variance, it cannot have a Gaussian distribution.

We finally compare directly with the centering and normalization of
Levin--Merzbach.  In the coupling above, the finite-clock and Haar-clock
contributions at level $m-1$ differ by at most one.  Together with
\eqref{eq:dominant-low-level-coupling}, this gives
\begin{equation}\label{eq:dominant-Haar-comparison}
 \|\cD_{B_m}-\mathcal P_m\|_2\leq C_2+1.
\end{equation}
Since $\E[\mathcal P_m]=M_m$ and
$\Var(\mathcal P_m)=V_m$, it follows that
\begin{equation}\label{eq:dominant-proxy-comparison}
 \E[\cD_{B_m}]=M_m+O(1),\qquad
 \left|\sqrt{\Var(\cD_{B_m})}-\sqrt{V_m}\right|=O(1).
\end{equation}

Since \(\sqrt{V_m}\to\infty\), we immediately get from 
\eqref{eq:dominant-proxy-comparison} that
\begin{equation}\label{from-exact-E_v-to-proxies}
 \frac{\sqrt{V_m}}{\sqrt{\Var(\cD_{B_m})}}\longrightarrow1,
 \qquad
 \frac{\E[\cD_{B_m}]-M_m}
      {\sqrt{\Var(\cD_{B_m})}}\longrightarrow0.
\end{equation}
Consequently, the preceding proxy-normalized convergence and Slutsky's
theorem, applied along \(m=n_k\), yield \eqref{eq:dominant-limit}.

\end{proof}

\section{Future work}\label{sec:future-work}

Several extensions of our results are postponed to future work.\newline  

The first one is a multivariate temporal theory with a common random
clock.  Fix \(q\geq2\), and let $Z,Y_1,\ldots,Y_q$
be independent random variables, each uniformly distributed on \([0,1)\).
Set \(N_H=\lfloor ZH\rfloor\) and
\begin{equation}\label{joint-intervals}
 X_{H,j}
 :=D_{N_H}(Y_j)
 =\sum_{n=0}^{N_H-1}
   \bigl(\mathbf 1_{[0,Y_j)}(T^n0)-Y_j\bigr),
 \qquad 1\leq j\leq q,
\end{equation}
and write
 $\mathbf X_H=(X_{H,1},\ldots,X_{H,q})$.
Conditional on \(N_H\), the coordinates of \(\mathbf X_H\) are independent
and identically distributed.  Consequently, setting 
 $\alpha_H
 :=\E\!\left[\Var(X_{H,1}\mid N_H)\right]$  and $\beta_H
 :=\Var\!\left(\E[X_{H,1}\mid N_H]\right)$,
then the covariance matrix has the 
form
\begin{equation}\label{covariance-matrix-joint-intervals}
 \Sigma_H:=\Cov(\mathbf X_H)
 =\alpha_H I_q+\beta_H\mathbf 1_q\mathbf 1_q^{\mathsf T}.
\end{equation}
Thus its diagonal entries are \(\alpha_H+\beta_H\), while every
off-diagonal entry is \(\beta_H\).  The goal would  be to  seek explicit
conditions implying a joint
temporal central limit theorem.  For example, assuming 
 $\rho_H:=\frac{\beta_H}{\alpha_H+\beta_H}\longrightarrow\rho$,
the expected coordinatewise-normalized limit is
\begin{equation}
 \frac{\mathbf X_H-\E[\mathbf X_H]}
      {\sqrt{\alpha_H+\beta_H}}
 \Longrightarrow
 \mathcal N_q\!\left(
   0,(1-\rho)I_q+\rho\mathbf 1_q\mathbf 1_q^{\mathsf T}
 \right).
\end{equation}
This vector formulation also leads naturally to functionals of the entire
endpoint-discrepancy process.  For instance,
\begin{equation}\label{endpoint-discrepancy-process}
 \frac1q\sum_{j=1}^qD_{N_H}(Y_j)
 \longrightarrow
 \int_0^1D_{N_H}(y)\,dy
\end{equation}
as \(q\to\infty\), conditionally on \(N_H\).
Similarly,
\begin{equation}\label{averaged-Lp-discrepancy}
 \frac1q\sum_{j=1}^q|D_{N_H}(Y_j)|^p
 \longrightarrow
 \int_0^1|D_{N_H}(y)|^p\,dy,
\end{equation}
providing a route to temporal limit theorems for averaged
\(L^p\)-discrepancy.\newline


Levin \cite{LevinMulti} has already proved a temporal central limit
theorem for a synchronized product of fixed-base odometers with distinct
prime bases, tested against a fixed multidimensional box.  
A different problem, closer to the present one-dimensional setting, is to
study the vector
\begin{equation}\label{two-clocks}
 \left(
  D^{(1)}_{\lfloor Z_1H_1\rfloor}(Y),
  D^{(2)}_{\lfloor Z_2H_2\rfloor}(Y)
 \right),
\end{equation}
where the two odometers are tested against the same random interval
\([0,Y)\).  When $Z_1$ and $Z_2$ are independent, the dependence is
created entirely by the shared endpoint, while when $Z_1=Z_2$, there is an
additional common-clock contribution.  The remaining questions are to
determine the cross-covariance, prove a joint temporal central limit
theorem, and extend the result from fixed pairwise coprime bases to
genuinely mixed-radix systems.\newline

A third question for future work concerns a  quenched temporal CLT for almost every fixed
endpoint $y$.

\section*{Acknowledgements}
The second author gratefully acknowledges the support of the Natural Sciences and Engineering Research Council of Canada  through the NSERC Discovery Grants  RGPIN-2022-04330. 

During the preparation of this work, the authors used ChatGPT (OpenAI, model GPT-5.6 Sol) to assist with literature and reference searches and with editing and organizing the manuscript, including revisions to LaTeX source originally written by the authors.
The core ideas underlying the work were developed in 2017 during the first author’s postdoctoral appointment at Queen’s University, well before general-purpose generative AI tools became available.
All AI-generated suggestions and proposed references were critically reviewed, and all mathematical content was independently verified by the authors, who take full responsibility for all results, proofs, and conclusions.

\end{document}